\documentclass [twoside,reqno, 12pt] {amsart}

\usepackage{tikz}

\usepackage{amsfonts}
\usepackage{amssymb}
\usepackage{a4}
\usepackage{color}

\newtheorem{thm}{Theorem}[section]

\newtheorem{lem}[thm]{Lemma}
\newtheorem{prop}[thm]{Proposition}

\newtheorem{rem}[thm]{Remark}

\theoremstyle{definition}

\numberwithin{equation}{section}

\renewcommand{\Re}{\hbox{Re}\,}
\renewcommand{\Im}{\hbox{Im}\,}

\newcommand{\C}{\mathbb{C}}

\renewcommand{\div}{\operatorname{div}}

\newcommand{\N}{\mathbb{N}}

\newcommand{\R}{\mathbb{R}}

\newcommand{\supp}{\operatorname{supp}}

\def\hat{\widehat}
\def\tilde{\widetilde}
\def \bfo {\begin {eqnarray*} }
\def \efo {\end {eqnarray*} }
\def \ba {\begin {eqnarray*} }
\def \ea {\end {eqnarray*} }
\def \beq {\begin {eqnarray}}
\def \eeq {\end {eqnarray}}
\def \supp {\hbox{supp }}

\def \dist {\hbox{dist}}

\def \p {\partial}

\def\hat{\widehat}
\def\tilde{\widetilde}
\def \bfo {\begin {eqnarray*} }
\def \efo {\end {eqnarray*} }
\def \ba {\begin {eqnarray*} }
\def \ea {\end {eqnarray*} }
\def \beq {\begin {eqnarray}}
\def \eeq {\end {eqnarray}}
\def \supp {\hbox{supp }}

\def \dist {\hbox{dist}}

\def \p {\partial}

\begin{document}

 \title[$L^p$ resolvent estimates]{On $L^p$ resolvent estimates for elliptic operators on compact manifolds}

\author[Krupchyk]{Katsiaryna Krupchyk}

\address
        {K. Krupchyk, Department of Mathematics and Statistics \\
         University of Helsinki\\
         P.O. Box 68 \\
         FI-00014   Helsinki\\
         Finland}

\email{katya.krupchyk@helsinki.fi}

\author[Uhlmann]{Gunther Uhlmann}

\address
       {G. Uhlmann, Department of Mathematics\\
       University of Washington\\
       Seattle, WA  98195-4350\\
       USA}
\email{gunther@math.washington.edu}

\maketitle

\begin{abstract}
We prove uniform $L^p$ estimates for resolvents of higher order elliptic self-adjoint differential operators on compact manifolds without boundary, generalizing a corresponding result of \cite{DKS_resolvent} in the case of  Laplace-- Beltrami operators on Riemannian manifolds. In doing so, we follow the methods, developed in  \cite{Bourgain_Shao_Sogge_Yao} very closely.   We also show that spectral regions in our $L^p$ resolvent estimates are optimal.

\end{abstract}

\section{Introduction and statement of results}

\label{sec_int}

The purpose of this paper is to extend the result of  \cite{DKS_resolvent}, see also \cite{Bourgain_Shao_Sogge_Yao},   for the Laplace-Beltrami operator $\Delta_g$ on a compact Riemannian manifold $(M,g)$ without boundary  of dimension $n\ge 3$,   to the case of higher order elliptic self-adjoint differential operators, and specifically to show how the methods of \cite{Bourgain_Shao_Sogge_Yao}  apply in this context.

In \cite{DKS_resolvent} it was established that given $\delta>0$ small,
there exists a constant $C=C(\delta)>0$ such that  for all $u\in C^\infty(M)$ and all $\zeta\in \mathcal{R}_\delta$, the following $L^p$ resolvent bound holds,
\begin{equation}
\label{eq_resolvent_est_laplacian}
\|u\|_{L^{\frac{2n}{n-2}}(M)}\le C\|(-\Delta_g-\zeta)u\|_{L^{\frac{2n}{n+2}}(M)},
\end{equation}
where
\[
\mathcal{R}_\delta=\{\zeta\in \C: (\Im \zeta)^2\ge 4\delta^2(\textrm{Re }\zeta+\delta^2)\}.
\]
Notice that $\mathcal{R}_\delta$ is the exterior of a parabolic region, containing the spectrum of $-\Delta_g$,   see Figure \ref{pic_laplacian}.
We observe that the bound \eqref{eq_resolvent_est_laplacian} cannot hold if $\mathcal{R}_\delta$ intersects the spectrum of $-\Delta_g$, as the latter is discrete. The interesting question, posed in \cite{DKS_resolvent} and subsequently studied in \cite{Bourgain_Shao_Sogge_Yao}, is how close $\mathcal{R}_\delta$ can come to the spectrum of $-\Delta_g$ near infinity, while still having the uniform estimate \eqref{eq_resolvent_est_laplacian}.

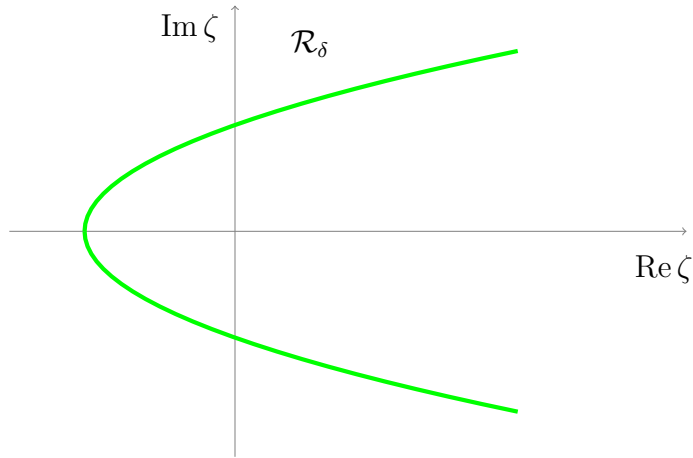
\begin{figure} [ht]
\centering
\begin{tikzpicture}
\draw [help lines, ->] (-3,0) -- (6,0);
\draw [help lines,  ->] (0,-3) -- (0,3);
\draw[green, ultra thick, domain=0:2.4] plot ({\x*\x-2}, \x);
\draw[green, ultra thick, domain=0:2.4] plot ({\x*\x-2}, -\x);
\node at (5.7,-0.5) {$\Re \zeta$};
\node at (-0.6, 2.7) {$\Im \zeta$};
\node at (1, 2.5) {$\mathcal{R}_\delta$};
\end{tikzpicture}
\caption{Spectral region $\mathcal{R}_\delta$ in the uniform resolvent bound \eqref{eq_resolvent_est_laplacian}.}
\label{pic_laplacian}
\end{figure}

Thanks to the work  \cite{Bourgain_Shao_Sogge_Yao}, we know that the region $\mathcal{R}_\delta$ is in general the maximal possible for the uniform estimate \eqref{eq_resolvent_est_laplacian} to hold. Indeed, in  \cite{Bourgain_Shao_Sogge_Yao} it is shown that the region  cannot be improved when $M$  is  the standard sphere, or more generally,  a Zoll manifold, due to   a cluster structure of the spectrum of $-\Delta_g$ on such manifolds, \cite{Weinstein_1977}. As explained in \cite{Bourgain_Shao_Sogge_Yao},  any sharpening in the spectral region is related to  improvements in estimates for the remainder term in the sharp Weyl law for $-\Delta_g$, which measures how uniformly its spectrum is distributed. Consequently, improvements in the spectral region $\mathcal{R}_\delta$ are available for manifolds of nonpositive  curvature and in the case of the torus with a flat metric, see \cite{Bourgain_Shao_Sogge_Yao}, and also \cite{Shen_2001}.

The corresponding uniform $L^p$ resolvent estimates for the standard Laplacian  on $\R^n$, $n\ge 3$, were obtained in  \cite{Kenig_Ruiz_Sogge}. Here in contrast to the case of a compact manifold, the estimates are valid  for all values of the complex spectral parameter $\zeta$.  In  \cite{Guillarmou_Hassell}  the results of \cite{Kenig_Ruiz_Sogge} were generalized to the case of non-trapping asymptotically conic manifolds.

To formulate our results let us begin by fixing some notation.  Let $M$ be a compact connected $C^\infty$ manifold without boundary of dimension $n\ge 2$, equipped with a strictly positive $C^\infty$ volume density $d\mu$.
Let $P$ be a differential operator on $M$ of order $m\ge 1$ with $C^\infty$ coefficients.  We assume that $P$ is elliptic and formally self-adjoint with respect to $d\mu$,
\[
\int_M P u\overline{v}d\mu=\int_M u\overline{Pv}d\mu,\quad u,v\in C^\infty(M).
\]
Let $p(x,\xi)\in C^\infty(T^*M)$ be the principal symbol of $P$, which is a real-valued homogeneous polynomial in $\xi$ of degree $m$.  Since $p(x,\xi)\ne 0$ for $\xi\ne 0$ and $T^*M\setminus\{0\}$ is connected, without loss of generality we shall assume, as we may,  that $p(x,\xi)>0$ for $\xi\ne 0$.  The order $m$ of the operator $P$ is therefore even.

If we equip the operator $P$ with the domain $C^\infty(M)$, $P$  becomes an unbounded symmetric essentially self-adjoint operator on $L^2(M)$, i.e. $P$ has a unique self-adjoint extension, which we shall denote again by $P$. The domain of the self-adjoint extension is $\mathcal{D}(P)=H^m(M)$, the standard Sobolev space on $M$.

An application of G\aa{}rding's inequality implies that there exists a constant $C>0$ such that $P\ge -CI$ in the sense of self-adjoint operators. Thus, after replacing $P$ by $P+CI$, we assume, as we may, that $P\ge 0$.

The spectrum of $P$ is discrete,   consisting only of real eigenvalues, where each eigenvalue is isolated and of finite multiplicity. Let $0\le \lambda_1\le\lambda_2\le \dots$ be the eigenvalues of $P$ repeated according to their multiplicity, and let $e_1,e_2,\ldots\in L^2(M)$ be the corresponding orthonormal basis of eigenfunctions.

Seeking to generalize \eqref{eq_resolvent_est_laplacian},  our goal is to find a region $\mathcal{R}\subset\C$, for which there holds a uniform $L^p$ bound of the form,
\begin{equation}
\label{eq_resolvent_est_0}
\|u\|_{L^q(M)}\le C_{\mathcal{R}}\|(P-\zeta)u\|_{L^p(M)},\quad u\in C^\infty(M),\quad \zeta\in\mathcal{R},
\end{equation}
for suitable $p$ and $q$.
Motivated by the classical Sobolev inequalities, we shall be interested in the estimate \eqref{eq_resolvent_est_0} for pairs $(p,q)$ belonging to the Sobolev line
\begin{equation}
\label{eq_p_q_1}
\frac{1}{p}-\frac{1}{q}=\frac{m}{n},
\end{equation}
assuming that $p<n/m$.
Following \cite{Bourgain_Shao_Sogge_Yao, DKS_resolvent}, we shall also require  the pairs $(p,q)$ to be on the duality line,
\begin{equation}
\label{eq_p_q_2}
\frac{1}{p}+\frac{1}{q}=1.
\end{equation}
The restrictions \eqref{eq_p_q_1} and \eqref{eq_p_q_2} imply that
\[
p=\frac{2n}{n+m},\quad q=\frac{2n}{n-m}, \quad n>m.
\]

It is clear that  the estimate \eqref{eq_resolvent_est_0} can only hold away from the spectrum of $P$.  Similarly to the case  of $-\Delta_g$, when establishing the estimate \eqref{eq_resolvent_est_0},   we shall in fact be concerned with the case of $\zeta$ away from all of $[0,\infty)$.  Given $\zeta\in  \C\setminus [0,\infty)$, it will then be convenient to write $\zeta=z^m$ with $z\in \Xi$, where
\[
\Xi=\{z\in \C: \arg (z)\in (0,2\pi/m)\}.
\]
This is due to that fact that  the map
\begin{align*}
f=f_m: \Xi\to \C\setminus [0,\infty), \quad z\mapsto z^{m},
\end{align*}
is a conformal isomorphism.  This map extends continuously to $f:\overline{\Xi}\to \C$ with $f(\p \Xi)=[0,\infty)$.

Notice that the region $\mathcal{R}_\delta$ in the uniform bound  \eqref{eq_resolvent_est_laplacian} satisfies
\[
\mathcal{R}_\delta=f_2(\Xi_\delta), \quad \Xi_\delta=\{z\in \C: \Im z\ge \delta\},
\]
By analogy with this, it is natural to try to establish the estimate \eqref{eq_resolvent_est_0} for $\zeta=z^m$, where
\[
z\in \Xi_\delta=\{z\in \Xi: \textrm{dist}(z,\p\Xi)\ge \delta\},
\]
with $\delta>0$ small but fixed. We have
\[
\Xi_\delta=\{z\in \C: \textrm{arg}(z)\in (0,2\pi/m), \Im z\ge \delta,-\Im(ze^{-2\pi i/m})\ge \delta\}.
\]


Associated with the principal symbol $p(x,\xi)$ of the operator $P$ is the cosphere
\[
\Sigma_x=\{\xi\in T_x^*M:p(x,\xi)=1\}, \quad x\in M.
\]
We may notice that for each  $x\in M$, the cosphere $\Sigma_x$ is a $C^\infty$ compact connected hypersurface in $\R^n$, see the discussion before Lemma \ref{lem_stationary_phase} below.
The cosphere $\Sigma_x$ is called strictly convex if the second fundamental form is  definite at each point of $\Sigma_x$.  This is equivalent to the fact that the Gaussian curvature of $\Sigma_x$ is non-vanishing.

The following theorem is the main result of this paper, which is a generalization of the uniform estimate \eqref{eq_resolvent_est_laplacian}, obtained in
\cite{DKS_resolvent}, to the case of higher order elliptic self-adjoint differential operators.
\begin{thm}
\label{thm_main}
Assume that $n>m\ge 2$ and that for each $x\in M$, the cosphere $\Sigma_x$ is strictly convex.  Then
given $\delta>0$ small, there is a constant $C=C(\delta)>0$ such that for all $u\in C^\infty(M)$ and all $z\in \Xi_\delta$, the following estimate holds
\begin{equation}
\label{eq_resolvent_est}
\|u\|_{L^{\frac{2n}{n-m}}(M)}\le C \|(P-z^m)u\|_{L^{\frac{2n}{n+m}}(M)}.
\end{equation}
\end{thm}

In the case of an elliptic operator $P$ of order $m\ge 4$, letting $\mathcal{R}_\delta=f(\Xi_\delta)$, a straightforward computation show that for $R>0$ sufficiently large, we have
\[
\mathcal{R}_\delta\cap\{\zeta\in \C:|\zeta|\ge R\}=(\mathcal{R}^+_\delta \cup \mathcal{R}^-_\delta)\cap \{\zeta\in \C:|\zeta|\ge R\},
\]
where
\begin{align*}
\mathcal{R}^+_\delta :=&\{\zeta\in \C: \Im \zeta\ge (\Re \zeta)^{\frac{m-1}{m}} m\delta+\mathcal{O}((\Re \zeta)^{\frac{m-3}{m}}), \Re \zeta\ge 0\}\\
&\cup
\{\zeta\in \C: \Im \zeta\le - (\Re\zeta)^{\frac{m-1}{m}} m\delta-\mathcal{O}((\Re\zeta)^{\frac{m-3}{m}}), \Re\zeta\ge 0\},
\end{align*}
and
\[
\mathcal{R}^-_\delta :=\{\zeta\in \C: \Re\zeta\le 0\}.
\]
Thus, for $|\zeta|$ sufficiently large, similarly to the case of $-\Delta_g$, the region $\mathcal{R}_\delta$ is the exterior of  a parabolic neighborhood of the spectrum of the operator $P$, see Figure \ref{pic_elliptic}.

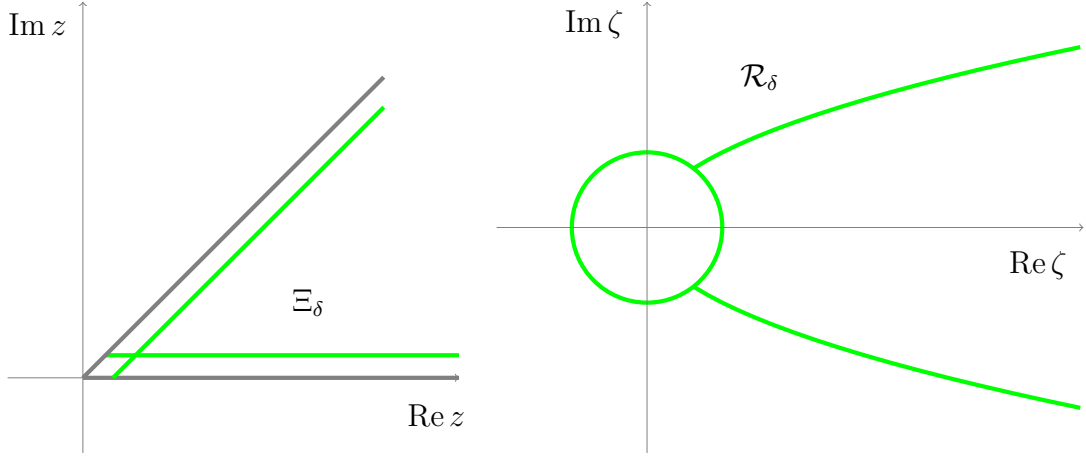
\begin{figure} [ht]
\centering
\begin{tikzpicture}
\draw [help lines, ->] (-1,0) -- (5,0);
\draw [help lines,  ->] (0,-1) -- (0,5);
\draw [green,  ultra thick, -] (0.3,0.3) -- (5,0.3);
\draw [help lines, ultra thick, -] (0,0) -- (4,4);
\draw [help lines, ultra thick, -] (0,0) -- (5,0);
\draw [green,  ultra thick, -] (0.4,0) -- (4,3.6);
\node at (4.7,-0.5) {$\Re z$};
\node at (-0.6, 4.7) {$\Im  z$};
\node at (3, 1) {$\Xi_\delta$};
\draw[green, ultra thick, domain=2:4.4] plot ({(\x-2)*(\x-2)+7.5}, \x);
\draw[green, ultra thick, domain=2:4.4] plot ({(\x-2)*(\x-2)+7.5}, 4-\x);
\draw [green,  ultra thick,  fill=white] (7.5,2) circle [radius=1];
\draw [help lines, ->] (5.5,2) -- (13.3,2);
\draw [help lines,  ->] (7.5,-1) -- (7.5,5);
\node at (12.7,1.5) {$\Re \zeta$};
\node at (6.8, 4.7) {$\Im \zeta$};
\node at (9, 4) {$\mathcal{R}_\delta$};
\end{tikzpicture}
\caption{The spectral regions $\Xi_\delta$ and $\mathcal{R}_\delta=f(\Xi_\delta)$ in the uniform estimate \eqref{eq_resolvent_est}.}
\label{pic_elliptic}
\end{figure}

As an example of an operator $P$ to which Theorem \ref{thm_main} applies, one can consider $P=(-\Delta_g)^k$,  $2k< n $, where $-\Delta_g$ is the Laplace--Beltrami operator on a compact Riemannian manifold $(M,g)$.

Our proof of Theorem \ref{thm_main} relies on the approach, developed in  \cite{Bourgain_Shao_Sogge_Yao}. The main ingredients here are the spectral cluster estimates, obtained in \cite{Sogge_1988} in the case of the Laplace--Beltrami operator on a compact Riemannian manifold, and in \cite{Seeger_Sogge_1989} in the case of higher order elliptic operators, the method of stationary phase, as well as the H\"ormander--Lax parametrix for  the operator $e^{it\sqrt[m]{P}}$ for small times.

Let us remark that the strict convexity of the cospheres $\Sigma_x$ in Theorem \ref{thm_main}  guarantees that the Fourier transform of the surface measure on $\Sigma_x$ has essentially the same decay at infinity, as that of the surface measure on the sphere, thanks to the method of stationary phase, see  \cite[Theorem 1.2.1, p. 50]{Sogge_book}.  This assumption also plays a crucial role in the derivation of the spectral cluster estimates for higher order elliptic operators  in \cite{Seeger_Sogge_1989}.

We may also notice that the a priori estimate \eqref{eq_resolvent_est} implies that the $L^2$ resolvent of $P$, $(P-\zeta)^{-1}$, $\zeta\in\C\setminus[0,\infty)$,  is a bounded operator:  $L^{\frac{2n}{n+m}}(M)\to L^{\frac{2n}{n-m}}(M)$, see Proposition \ref{eq_prop_resolvent_on_L_p} below.

Our next result shows that the region $\Xi_\delta$  in \eqref{eq_resolvent_est}  is in general optimal for higher order elliptic  operators, since it cannot be improved for an operator  whose principal symbol has a periodic Hamilton flow.   This is due to the fact that the spectrum of such an operator is distributed in a non-uniform fashion, displaying a cluster  structure, see \cite{Colin_de_Verdiere_1979} and \cite{Weinstein_1977}.

\begin{thm}
\label{thm_main_2}
Assume that $n>m\ge 2$ and that for each $x\in M$, the cosphere $\Sigma_x$ is strictly convex. Assume furthermore that
  the subprincipal symbol of the operator $P$ vanishes, and that  the  Hamilton flow of the principal symbol $p$ is periodic, with a common minimal period on $p^{-1}(1)$.  Then  there exist

(i) a sequence $z_k\in \Xi$ such that $\emph{\Re} z_k\to \infty$,  $0<\emph{\Im} z_k\to 0$ as $k\to \infty$, and
\[
\|(P-z_k^m)^{-1}\|_{L^{\frac{2n}{n+m}}(M)\to L^{\frac{2n}{n-m}}(M)}\to \infty,\quad k\to \infty,
\]
and

(ii) a sequence  $z_k \in \Xi$ such that  $\emph{\Re} (z_ke^{-2\pi i/m})\to \infty$, $0<-\emph{\Im} (z_ke^{-2\pi i/m})\to 0$ as $k\to \infty$, and
\[
\|(P-z_k^m)^{-1}\|_{L^{\frac{2n}{n+m}}(M)\to L^{\frac{2n}{n-m}}(M)}\to \infty,\quad k\to \infty.
\]
\end{thm}

As an example of the operator $P$ in Theorem \ref{thm_main_2} we can take $P=(-\Delta_g)^{k}$, $2k<n$,  on a Zoll manifold $M$,  similarly to the case when $k=1$ in \cite{Bourgain_Shao_Sogge_Yao}.
To prove Theorem \ref{thm_main_2} we shall also use the methods of  \cite{Bourgain_Shao_Sogge_Yao}.

The paper is organized as follows.  Section \ref{sec_proof_thm_1} is devoted to the proof of Theorem \ref{thm_main} while Section  \ref{sec_proof_thm2} contains the proof of Theorem \ref{thm_main_2}.

\section{Proof of Theorem \ref{thm_main}}

\label{sec_proof_thm_1}

\subsection{Formula for the resolvent $(P-z^m)^{-1}$ based on a half wave group for  $P^{1/m}$}
We shall denote by $\Psi^\mu_{\textrm{cl}}(M)$  the space of classical pseudodifferential operators of order $\mu$ on $M$.  Let $Q=P^{1/m}$ be defined by the spectral theorem. According to Seeley's theorem,  see \cite[Theorem 3.3.1]{Sogge_book}, we have $Q\in \Psi^1_{\textrm{cl}}(M)$ with the principal symbol $q=p^{1/m}$. Furthermore, $\mathcal{D}(Q)=H^1(M)$, and the eigenvalues of $Q$ are $\mu_j=\lambda_j^{1/m}$, $j=1,2,\dots$.

Letting $z\in \Xi$ and  following \cite{Bourgain_Shao_Sogge_Yao},  let us derive a natural formula for the $L^2$ resolvent $(P-z^m)^{-1}$.  To that end, we write
$(P-z^m)^{-1}=m_z(Q)$,
where the multiplier $m_z(Q)$ is given by $m_z(\tau)=(\tau^m-z^m)^{-1}$. Using the inverse Fourier transform, we get
\[
m_z(\tau)=\frac{1}{2\pi}\int_{-\infty}^{+\infty}\hat{m_z}(t) e^{it \tau}dt, \quad \hat{m_z}(t)=\int_{-\infty}^{+\infty} \frac{1}{\tau^m-z^m}e^{-it\tau}d\tau.
\]
We shall need the following result.
\begin{lem}
\label{lem_residue}
Let $z\in \Xi$. Then for any $t\in\R$, we have
\begin{equation}
\label{eq_residue}
\int_{-\infty}^{+\infty} \frac{1}{\tau^m-z^m}e^{-it\tau}d\tau=\frac{2\pi i}{mz^{m-1}}\sum_{k=0}^{m/2-1} e^{2\pi k i/m+i|t|\tau_k},
\end{equation}
where $\tau_k= ze^{2\pi k i/m}$, $k=0,1,\dots, m/2-1$. Here $\emph{\Im} \tau_k>0$,   $k=0,1,\dots, m/2-1$.
\end{lem}

\begin{proof}
To show \eqref{eq_residue} we shall use the residue calculus. To that end writing $z=|z| e^{i\varphi}$, $0<\varphi<2\pi/m$, we obtain that the poles of the rational function $\C\ni\tau\mapsto (\tau^m-z^m)^{-1}$ are given by
\[
\tau_k=|z|e^{i(m\varphi+2\pi k)/m}=ze^{2\pi k i/m},\quad k=0,\dots, m-1.
\]
Notice that the poles are simple, none of them are on the real line,  the poles $\tau_k$, $k=0,\dots, m/2-1$, are in the upper half plane, and the poles $\tau_k$,  $k=m/2, \dots, m-1$, are in the lower half plane.

We have $|e^{-it\tau}|=e^{t \textrm{Im}\tau}$.  Let first $t\le 0$. Deforming the contour of integration in the upper half plane, we get
\begin{align*}
\int_{-\infty}^{+\infty} \frac{1}{\tau^m-z^m}e^{-it\tau}d\tau&=2\pi i \sum_{k=0}^{m/2-1} \textrm{Res}\bigg( \frac{e^{-it\tau}}{\tau^m-z^m}; \tau_k \bigg)=2\pi i \sum_{k=0}^{m/2-1} \frac{e^{-it\tau_k}}{m\tau_k^{m-1}}\\
&=\frac{2\pi i}{mz^{m-1}} \sum_{k=0}^{m/2-1}  e^{2\pi k i/m-it\tau_k},\quad t\le 0.
\end{align*}
Let now $t>0$. Then by deforming the contour of integration in the lower half plane, we conclude that
\begin{align*}
\int_{-\infty}^{+\infty} \frac{1}{\tau^m-z^m}&e^{-it\tau}d\tau=-2\pi i \sum_{k=m/2}^{m-1} \textrm{Res}\bigg( \frac{e^{-it\tau}}{\tau^m-z^m}; \tau_k \bigg)=-2\pi i \sum_{k=m/2}^{m-1} \frac{e^{-it\tau_k}}{m\tau_k^{m-1}}\\
&=-\frac{2\pi i}{mz^{m-1}} \sum_{k=m/2}^{m-1}  e^{2\pi k i/m-it\tau_k}= -\frac{2\pi i}{mz^{m-1}} \sum_{k=0}^{m/2-1}  e^{\pi i}e^{2\pi k i/m-it\tau_{m/2+k}}\\
&=\frac{2\pi i}{mz^{m-1}} \sum_{k=0}^{m/2-1}  e^{2\pi k i/m+it\tau_k},\quad t> 0.
\end{align*}
Thus, \eqref{eq_residue} follows. The proof of Lemma \ref{lem_residue} is complete.
\end{proof}

Let $z\in \Xi$. Then by \eqref{eq_residue}, we obtain that
\[
m_z(\tau)=\frac{ i}{mz^{m-1}}\sum_{k=0}^{m/2-1} e^{2\pi k i/m} \int_{-\infty}^{+\infty} e^{i|t|\tau_k+it\tau}d t.
\]
Therefore, we have the following formula for the resolvent of $P$,
\begin{equation}
\label{eq_resolvent_formula}
(P-z^m)^{-1}=m_z(Q)=\frac{ i}{mz^{m-1}}\sum_{k=0}^{m/2-1} e^{2\pi k i/m} \int_{-\infty}^{+\infty} e^{i|t|\tau_k} e^{itQ}d t.
\end{equation}
Here $\tau_k= ze^{2\pi k i/m}$ and $\Im \tau_k>0$, $k=0,1,\dots, m/2-1$.

\subsection{Consequences of the spectral projection estimates}

Assume that, for each $x\in M$, the cosphere $\Sigma_x=\{\xi\in T_x^*M: q(x,\xi)=1\}$ is strictly convex. Consider the $k$'th spectral cluster of the operator $Q$,
\[
\{\mu_j\in \textrm{spec}(Q):\mu_j\in [k-1,k)\},
\]
and denote by $\chi_k$ the spectral projection operator on the space, generated by the eigenfunctions, corresponding to the $k$th spectral cluster,
\[
\chi_k f=\sum_{\mu_j\in [k-1,k)} E_jf,\quad f\in C^\infty(M).
\]
Here $E_j:L^2(M)\to L^2(M)$ is the orthogonal projection onto the space, spanned by  $e_j$, i.e.
\[
E_jf(x)=\bigg(\int_M f(y)\overline{e_j(y)}d\mu(y)\bigg) e_j(x).
\]

It was shown in  \cite{Seeger_Sogge_1989}, see also \cite[Theorem 5.1.1]{Sogge_book}, that for $p\ge \frac{2(n+1)}{n-1}$, we have
\begin{equation}
\label{eq_spectral_cluster}
\|\chi_k\|_{L^2(M)\to L^p(M)}\le Ck^{\sigma(p)},\quad \sigma(p)=n\bigg(\frac{1}{2}-\frac{1}{p}\bigg)-\frac{1}{2},
\end{equation}
where $C>0$ is a constant,  and the dual estimate,
\begin{equation}
\label{eq_spectral_cluster_2}
\|\chi_k\|_{L^{p'}(M)\to L^2(M)}\le Ck^{\sigma(p)}, \quad p'=\frac{p}{p-1}.
\end{equation}

Similarly to \cite[Lemma 2.3]{Bourgain_Shao_Sogge_Yao}, we have the following consequence of the spectral clusters estimates \eqref{eq_spectral_cluster} and \eqref{eq_spectral_cluster_2}.

\begin{lem}
\label{lem_truncated_eq}
Assume that, for each $x\in M$, the cosphere $\Sigma_x=\{\xi\in T_x^*M: q(x,\xi)=1\}$ is strictly convex. Let $\alpha\in C([0,\infty),\C)$ and define the operators $\alpha_k(Q)$ as follows,
\[
\alpha_k(Q)f=\sum_{\mu_j\in [k-1,k)}\alpha(\mu_j) E_j f, \quad f\in C^\infty(M),
\]
$k=1,2,\dots$. Then if $p\ge \frac{2(n+1)}{n-1}$, we get
\begin{equation}
\label{eq_truncated}
\|\alpha_k(Q)f\|_{L^p(M)}\le C k^{2\sigma(p)} (\sup_{\tau\in [k-1;k)}|\alpha(\tau)|)\|f\|_{L^{\frac{p}{p-1}}(M)},\ \sigma(p)=n\bigg(\frac{1}{2}-\frac{1}{p}\bigg)-\frac{1}{2},
\end{equation}
where $C>0$ is a constant independent of the function $\alpha$.
\end{lem}

\begin{proof}
First notice that $\alpha_k(Q)=\chi_k\circ\alpha_k(Q)$. Let $p\ge \frac{2(n+1)}{n-1}$. Then using  the spectral clusters estimates \eqref{eq_spectral_cluster} and \eqref{eq_spectral_cluster_2}, we obtain that
\begin{align*}
\|\alpha_k(Q)f\|_{L^p(M)}&\le C k^{\sigma(p)}\|\alpha_k(Q)f\|_{L^2(M)}\\
&=C k^{\sigma(p)}\bigg( \sum_{\mu_j\in [k-1,k)}|\alpha(\mu_j)|^2\|E_jf\|_{L^2(M)}^2  \bigg)^{1/2}\\
&\le C k^{\sigma(p)} (\sup_{\tau\in [k-1,k)}|\alpha(\tau)|) \bigg( \sum_{\mu_j\in [k-1,k)}\|E_jf\|_{L^2(M)}^2  \bigg)^{1/2}\\
&=  C k^{\sigma(p)} (\sup_{\tau\in [k-1,k)}|\alpha(\tau)|) \|\chi_kf\|_{L^2(M)}\\
&\le C k^{2\sigma(p)} (\sup_{\tau\in [k-1,k)}|\alpha(\tau)|) \|f\|_{L^\frac{p}{p-1}(M)}.
\end{align*}
\end{proof}

\begin{lem}
\label{lem_based_on_littlewood_paley}
Assume that  for each $x\in M$, the cosphere $\Sigma_x=\{\xi\in T_x^*M: q(x,\xi)=1\}$ is strictly convex. Let $\alpha\in C([0,\infty),\C)$ be such that
\begin{equation}
\label{eq_A_sup}
A=\sup_{\tau\in[0,\infty)}(1+\tau^m)|\alpha(\tau)|<\infty.
\end{equation}
 Then we have
\begin{equation}
\label{eq_Littlewood_Paley}
\|\alpha(Q)f\|_{L^{\frac{2n}{n-m}}(M)}\le C  A\|f\|_{L^{\frac{2n}{n+m}}(M)},
\end{equation}
where $\alpha(Q)$ is the operator defined by
\[
\alpha(Q)f=\sum_{j=1}^\infty\alpha(\mu_j) E_j f, \quad f\in C^\infty(M),
\]
and $C>0$ is a constant independent of the function $\alpha$.
\end{lem}

\begin{proof}
To establish \eqref{eq_Littlewood_Paley}, we shall follow  \cite[Lemma 2.3]{Bourgain_Shao_Sogge_Yao}, see also \cite{Kenig_Ruiz_Sogge}, and use the one dimensional Littlewood--Paley theory. To that end, let
\[
\chi(t)=\begin{cases} 1,& t\in [1/2,1),\\
0,& t\notin [1/2,1),
\end{cases}
\]
be the characteristic function of the interval $[1/2,1)$.  Setting $\chi_j(\tau)=\chi(2^{-j}\tau)$, we obtain the dyadic partition  of unity in $[0,\infty)$,
$\chi_0(\tau)+\sum_{j=1}^\infty\chi_j(\tau)=1$, where $\chi_0(\tau)=1$ when $\tau\in [0,1)$, and $\chi_0(\tau)=0$ otherwise.

Define $\alpha_j(\tau)=\alpha(\tau)\chi_j(\tau)$, $j=0,1,\dots$. Assume that we have proved that
\begin{equation}
\label{eq_Littlewood_Paley_j}
\|\alpha_j(Q)f\|_{L^{\frac{2n}{n-m}}(M)}\le S \|f\|_{L^{\frac{2n}{n+m}}(M)},\quad j=0,1,\dots,
\end{equation}
with some constant $S>0$.
By the Littlewood--Paley theorem and Minkowski's inequality, we conclude from \eqref{eq_Littlewood_Paley_j} that
\begin{equation}
\label{eq_Littlewood_Paley_inside}
\|\alpha(Q)f\|_{L^{\frac{2n}{n-m}}(M)}\le C_{q,p}  S\|f\|_{L^{\frac{2n}{n+m}}(M)},
\end{equation}
where $C_{q,p}>0$ depends on $q$ and $p$ only,
see  \cite{Kenig_Ruiz_Sogge} and \cite{Ruiz_lecture_notes}.
Let us present these arguments for the convenience of the reader.  We shall write  $p=\frac{2n}{n+m}$ and $q=\frac{2n}{n-m}$. Then $1<p< 2< q$. As $q>1$, by Littlewood--Paley theorem, we get
\begin{align*}
\|\alpha(Q)f\|_{L^q(M)}&\le C_q\bigg\|\bigg(\sum_{j=0}^\infty |\alpha_j(Q)f|^2\bigg)^{1/2}\bigg\|_{L^q(M)}\\
&=C_q \bigg\|\sum_{j=0}^\infty |\alpha_j(Q)f|^2\bigg\|_{L^{q/2}(M)}^{1/2}:=I_1.
\end{align*}
As $q/2\ge 1$, we may write from Minkowski's inequality that
\begin{align*}
I_1\le C_q \bigg(\sum_{j=0}^\infty \| |\alpha_j(Q)f|^2\|_{L^{q/2}(M)}\bigg)^{1/2}= C_q \bigg(\sum_{j=0}^\infty \| \alpha_j(Q)f\|_{L^{q}(M)}^2\bigg)^{1/2}:=I_2.
\end{align*}
As $\chi_j=\chi_j^2$, $j=0,1,\dots$, it follows from \eqref{eq_Littlewood_Paley_j} that
\begin{align*}
I_2&\le C_q S \bigg(\sum_{j=0}^\infty \| \chi_j(Q)f\|_{L^{p}(M)}^2\bigg)^{1/2}\\
&=C_qS\bigg( \bigg\| \bigg\{\int_M |\chi_j(Q)f(x)|^pd\mu(x)\bigg\}\bigg\|_{l^{2/p}}  \bigg)^{1/p}:=I_3,
\end{align*}
where $\|\{a_j\}\|_{l^{2/p}}$ denotes the $l^{2/p}$--norm of the sequence $\{a_j\}$. Since $2/p>1$, by Minkowski's inequality,
\begin{align*}
I_3&\le C_qS\bigg( \int_M \| \{ |\chi_j(Q)f|^p\} \|_{l^{2/p}} d\mu  \bigg)^{1/p} =C_qS \bigg\| \bigg(\sum_{j=0}^\infty |\chi_j(Q)f|^2\bigg)^{1/2}  \bigg\|_{L^p(M)}\\
&\le C_qC_pS \|f\|_{L^p(M)},
\end{align*}
which shows \eqref{eq_Littlewood_Paley_inside}.

Thus, we are left with proving \eqref{eq_Littlewood_Paley_j}. Let $f\in C^\infty(M)$. For $j=1,2,\dots$, we write
\begin{align*}
\alpha_j(Q) f&=\sum_{l=1}^\infty \alpha_j(\mu_l)E_l f=\sum_{\mu_l\in [2^{j-1},2^j)} \alpha_j(\mu_l)E_l f\\
&=\sum_{r=1}^{2^j-2^{j-1}} \sum_{\mu_l\in [2^{j-1}+r-1, 2^{j-1}+r) } \alpha_j(\mu_l)E_l f=\sum_{r=1}^{2^{j-1}} \alpha_{j,2^{j-1}+r}(Q)f,
\end{align*}
where the truncated operator $\alpha_{j,k}(Q)$ is given by
\[
\alpha_{j,k}(Q)f=\sum_{\mu_l\in [k-1,k)} \alpha_j(\mu_l)E_l f.
\]

Since $\frac{2n}{n-m}\ge \frac{2(n+1)}{n-1}$, by \eqref{eq_truncated} and the fact that $\sigma(2n/(n-m))=(m-1)/2$, we get
\begin{align*}
\|&\alpha_j(Q)f\|_{L^{\frac{2n}{n-m}}(M)}\le \sum_{r=1}^{2^{j-1}} \|\alpha_{j,2^{j-1}+r}(Q)f\|_{L^{\frac{2n}{n-m}}(M)} \\
&\le C \sum_{r=1}^{2^{j-1}} (2^{j-1}+r)^{m-1}(\sup_{\tau\in [2^{j-1}+r-1, 2^{j-1}+r)} |\alpha(\tau)|)\|f\|_{L^{\frac{2n}{n+m}}(M)},\quad j=1,2,\dots.
\end{align*}
Now using \eqref{eq_A_sup}, we obtain that
\begin{equation}
\label{eq_littlewood_parts_1}
\begin{aligned}
\|\alpha_j(Q)f\|_{L^{\frac{2n}{n-m}}(M)}&\le C A\sum_{r=1}^{2^{j-1}}(2^{j-1}+r)^{m-1} \frac{1}{(2^{j-1}+r-1)^m}\|f\|_{L^{\frac{2n}{n+m}}(M)}\\
&\le CA \sum_{r=1}^{2^{j-1}} \frac{(2^{j-1}2)^{m-1}}{(2^{j-1})^m} \|f\|_{L^{\frac{2n}{n+m}}(M)}
\le CA \|f\|_{L^{\frac{2n}{n+m}}(M)},
\end{aligned}
\end{equation}
for $ j=1,2,\dots$.
We also have
\[
\alpha_0(Q)f=\sum_{\mu_l\in [0,1)}\alpha(\mu_l)E_l f,
\]
and therefore, it follows from  \eqref{eq_truncated} that
\begin{equation}
\label{eq_littlewood_parts_2}
\|\alpha_0(Q)f\|_{L^{\frac{2n}{n-m}}(M)}\le C (\sup_{\tau\in [0,1)} |\alpha(\tau)|)\|f\|_{L^{\frac{2n}{n+m}}(M)}\le CA \|f\|_{L^{\frac{2n}{n+m}}(M)}.
\end{equation}
We obtain \eqref{eq_Littlewood_Paley_j} as a consequence of  \eqref{eq_littlewood_parts_1} and \eqref{eq_littlewood_parts_2}. The proof of Lemma \ref{lem_based_on_littlewood_paley} is complete.
\end{proof}

\subsection{Derivation of the resolvent estimate with bounded $|z|$}
Let us first prove the resolvent estimate \eqref{eq_resolvent_est} for all  $z\in \Xi_\delta$ when $|z|$ is bounded by a fixed constant, i.e. $
z\in \Xi_{\delta}\cap \{z\in \C:|z|\le D\}$.
To that end, consider the multiplier
\[
m_z(\tau)=\frac{1}{\tau^m-z^m},\quad \tau\in [0,\infty).
\]
First notice that $\tau^m-z^m\ne 0$ for all $\tau\ge 0$ and all $z\in \C$ with $\arg(z)\in (0,2\pi/m)$. Then by continuity of $|\tau^m-z^m|$ on a compact set, we have that  for any $A, D, \delta>0$, there exists a constant $C>0$ such that   $|\tau^m-z^m|\ge 1/C$ for  $\tau\in [0,A]$ and $z\in \Xi_{\delta}\cap \{z\in \C:|z|\le D\}$.
For $\tau$ large and $z\in \Xi_{\delta}\cap \{z\in \C:|z|\le D\}$, we have $|\tau^m-z^m|\sim \tau^m$, and therefore, we conclude that
\[
|m_z(\tau)|\le C_{\delta,D}(1+\tau^m)^{-1}
\]
uniformly in $z\in \Xi_{\delta}\cap \{z\in \C:|z|\le D\}$.  By appealing to Lemma \ref{lem_based_on_littlewood_paley}, we obtain the resolvent estimate \eqref{eq_resolvent_est} for $z\in \Xi_{\delta}\cap \{z\in \C:|z|\le D\}$.

\begin{rem}
\label{rem_non_uniform_est}
Notice that applying Lemma \ref{lem_based_on_littlewood_paley}, we can immediately obtain the (non-uniform) estimate
\[
\|u\|_{L^{\frac{2n}{n-m}}(M)}\le C_\zeta \|(P-\zeta)u\|_{L^{\frac{2n}{n+m}}(M)},
\]
for all $\zeta\in \C\setminus [0,\infty)$ and $u\in C^\infty(M)$.
\end{rem}

\subsection{Uniform bounds for a local term in the case of  unbounded $|z|$}
Let $z\in \Xi_{\delta}\cap \{z\in \C:|z|\ge 1\}$. Here it will be convenient to use the representation \eqref{eq_resolvent_formula} for the multiplier $m_z(Q)$.
To define the localized version of $m_z(Q)$, we fix a function $\rho\in C^\infty(\R)$ satisfying
\begin{equation}
\label{eq_def_rho_eps}
\rho(t)=\begin{cases}
1, & |t|\le \varepsilon/2,\\
0, & |t|\ge \varepsilon,
\end{cases}
\end{equation}
where $0<\varepsilon<1/2$ will be specified later.  In view of \eqref{eq_resolvent_formula}, the localized version of $m_z(Q)$ is given by
\begin{equation}
\label{eq_m_z_loc}
m_z^{\textrm{loc}}(Q)f=\frac{ i}{mz^{m-1}}\sum_{k=0}^{m/2-1} e^{2\pi k i/m} \int_{-\infty}^{+\infty} \rho(t)e^{i|t|\tau_k} e^{itQ}fd t,\quad  f\in C^\infty(M).
\end{equation}
Here $\tau_k= ze^{2\pi k i/m}$ and $\Im \tau_k>0$, $k=0,1,\dots, m/2-1$.

To prove the resolvent estimate \eqref{eq_resolvent_est} for $z\in \Xi_{\delta}\cap \{z\in \C:|z|\ge 1\}$, let us first establish
this estimate for $m_z^{\textrm{loc}}(Q)$, i.e.
\begin{equation}
\label{eq_resolvent_with_b_loc}
\| m_z^{\textrm{loc}}(Q)f\|_{L^{\frac{2n}{n-m}}(M)}\le C\|f\|_{L^\frac{2n}{n+m}(M)}.
\end{equation}
When doing so we shall use a dyadic partition of the $t$--interval in the definition \eqref{eq_m_z_loc} of $m_z^{\textrm{loc}}(Q)$. To that end let $\psi\in C^\infty_0(\R)$ be such that $\supp(\psi)\subset [-2,2]$, $\psi=1$ on $[-1,1]$, and  $\psi$ is even. Define $\beta(t)=\psi(t)-\psi(2t)$. Thus,
\[
 \beta(t)=0,\quad |t|\notin [1/2,2],
 \]
 and
\[
\sum_{j=-\infty}^{+\infty} \beta (2^{-j}t)=1,\quad t\ne 0.
\]
It will be convenient to write,
\[
\tilde \rho(t)=1-\sum_{j=0}^{+\infty} \beta (2^{-j}t)\in C^\infty_0(\R).
\]
Notice that $\tilde \rho(t)=0$ when $|t|\ge 1$.

For a given $z\in \Xi_{\delta}\cap \{z\in \C:|z|\ge 1\}$, we define the multipliers
\begin{equation}
\label{eq_S_j}
S_{z,j}(\tau)=\frac{ i}{mz^{m-1}}\sum_{k=0}^{m/2-1} e^{2\pi k i/m} \int_{-\infty}^{+\infty} \beta (2^{-j}|z|t) \rho(t)e^{i|t|\tau_k} e^{it\tau}d t,\quad  j=0,1,2,\dots,
\end{equation}
and
\begin{equation}
\label{eq_S_0}
\tilde S_{z}(\tau)=\frac{ i}{mz^{m-1}}\sum_{k=0}^{m/2-1} e^{2\pi k i/m} \int_{-\infty}^{+\infty} \tilde\rho (|z|t) \rho(t)e^{i|t|\tau_k} e^{it\tau}d t.
\end{equation}

We have
\begin{equation}
\label{eq_S_j=0}
S_{z,j}=0 \quad \text{if}\quad 2^{-j}|z|\le 1.
\end{equation}
Indeed, if $|t|\le \varepsilon$, then $2^{-j}|z||t|<1/2 $ and therefore, $\beta(2^{-j}|z|t)=0$.

The bound \eqref{eq_resolvent_with_b_loc} follows once we show that there is a uniform constant $C$ so that for all $z\in \Xi_{\delta}\cap \{z\in \C:|z|\ge 1\}$, we have
\begin{equation}
\label{eq_estim_S_j}
\|S_{z,j}(Q)f\|_{L^{\frac{2n}{n-m}}(M)}\le C 2^{j\frac{2n-m-nm}{2n}}\|f\|_{L^\frac{2n}{n+m}(M)},\quad  j=0,1,\dots,
\end{equation}
and
\begin{equation}
\label{eq_estim_S_0}
\| \tilde S_{z}(Q)f\|_{L^{\frac{2n}{n-m}}(M)}\le C\|f\|_{L^\frac{2n}{n+m}(M)}.
\end{equation}
Let us start with establishing the estimate \eqref{eq_estim_S_0}.  When doing so, we shall follow \cite{Shao_Yao} and obtain the following result.

\begin{lem}
\label{lem_symbol}
The multiplier $\tilde S_{z}$ belongs to the symbol class $ S^{-m}(\R)$ uniformly in $z\in \C$, $|z|\ge 1$, i.e.
\begin{equation}
\label{eq_estim_symbol_S_0}
|d_\tau^j \tilde S_{z}(\tau)|\le C_j(1+|\tau|)^{-m-j},\quad j=0,1,2,\dots,
\end{equation}
with the constants $C_j$ independent of $z$.
\end{lem}

\begin{proof}
Recall first that $ \tilde\rho (|z|t)=0$ when $|t|\ge 1/|z|$.  Furthermore, as $\Im\tau_k>0$, $k=0,1,\dots,m/2-1$, we conclude that $|e^{i|t|\tau_k}|\le 1$.

Let $|\tau|\le 1$. Then for $j=0,1,\dots$, we have
\begin{align*}
|d_\tau^j \tilde S_{z}(\tau)|\le \frac{C}{|z|^{m-1}}\int_{-1/|z|}^{1/|z|}|t|^j dt\le \frac{C}{|z|^{m+j}}\le C,
\end{align*}
uniformly in $z$, $|z|\ge 1$,  which shows the estimate \eqref{eq_estim_symbol_S_0} in the case  $|\tau|\le 1$.

Assume now that $|\tau|> 1$. Let us first prove the estimate \eqref{eq_estim_symbol_S_0} for $j=0$. To that end we shall integrate by parts $m$ times in the expression \eqref{eq_S_0} for $\tilde S_{z}$.

Let us first explain that all boundary terms vanish when we integrate by parts $m-1$ times in \eqref{eq_S_0}. Indeed, integrating by parts once  in \eqref{eq_S_0}, we obtain the following boundary
terms,
\begin{align*}
\frac{ i}{i\tau mz^{m-1}}\sum_{k=0}^{m/2-1} e^{2\pi k i/m}   \bigg( \tilde\rho (|z|t) \rho(t)e^{-it\tau_k} e^{it\tau}|_{t=-\infty}^{t=0}+ \tilde\rho (|z|t) \rho(t)e^{it\tau_k} e^{it\tau}|_{t=0}^{t=+\infty}\bigg)\\
=\frac{ i}{i\tau mz^{m-1}}\sum_{k=0}^{m/2-1} e^{2\pi k i/m}   \bigg(  1-1\bigg)=0.
\end{align*}
Here we have used the fact that $\tilde \rho$ and $\rho$ are compactly supported, and $\tilde \rho(0)=\rho(0)=1$.

Furthermore, since all  the derivatives of $\tilde\rho$ and $\rho$ vanish at the origin, when integrating by parts $m$ times in \eqref{eq_S_0},
the only possible contribution to the boundary terms may be written in the form $\sum_{l=1}^{m}B_l$, where
\begin{align*}
B_l=\frac{ i}{(i\tau)^l mz^{m-1}}&\sum_{k=0}^{m/2-1} e^{2\pi k i/m}  (-1)^{l-1} \bigg(\tilde\rho (|z|t) \rho(t)(-i\tau_k)^{l-1}e^{-it\tau_k} e^{it\tau}|_{t=-\infty}^{t=0}\\
&+ \tilde\rho (|z|t) \rho(t)(i\tau_k)^{l-1}e^{it\tau_k} e^{it\tau}|_{t=0}^{t=+\infty}\bigg)\\
&=\frac{ i}{(i\tau)^l mz^{m-1}}\sum_{k=0}^{m/2-1} e^{2\pi k i/m}  (-1)^{l-1} ( (-i\tau_k)^{l-1}- (i\tau_k)^{l-1}).
\end{align*}
When $l$ is odd, it is clear that $B_l=0$. Recall now that $m$ is even.  When $l$ is even and $l\ne m$, we also have $B_l=0$ due to the fact that
\begin{align*}
\sum_{k=0}^{m/2-1} e^{2\pi k i/m} (\tau_k)^{l-1}=z^{l-1} \sum_{k=0}^{m/2-1} (e^{2\pi  l i/m})^k=z^{l-1}\frac{1-e^{\pi li}}{1-e^{2\pi li/m}}=0.
\end{align*}
 Here we have used that $\tau_k=ze^{2\pi k i/m}$ and the fact that $e^{2\pi li/m}\ne 1$ when $2\le l\le m-2$.  Hence, when integrating by parts $m$ times in \eqref{eq_S_0},
the only possible contribution to the boundary terms is of the form,
\begin{equation}
\label{eq_bound_b_m}
\begin{aligned}
B_m=\frac{2}{\tau^m m z^{m-1}}\sum_{k=0}^{m/2-1} e^{2\pi k i/m} (\tau_k)^{m-1}=\frac{2}{\tau^m m}\sum_{k=0}^{m/2-1} e^{2\pi k i}=\frac{1}{\tau^m}.
\end{aligned}
\end{equation}

Let us explain how to estimate the integrals arising after having integrated by parts $m$ times in \eqref{eq_S_0}.
The worst case scenario occurs when no derivatives fall on $\rho(t)$, and the corresponding contribution can be estimated by  a constant times
\begin{equation}
\label{eq_interior_int}
\bigg|\frac{1}{\tau^m}\int_{-1/|z|}^0 |z|^{l_1} (d_t^{l_1}\tilde\rho)(|z|t)\rho(t)(-i\tau_k)^{l_2}e^{-it\tau_k}e^{it\tau}dt\bigg|\le C\frac{|z|^{m-1}}{|\tau|^m}.
\end{equation}
Here  $l_1+l_2=m$.
Then it follows from \eqref{eq_S_0}, \eqref{eq_interior_int} and \eqref{eq_bound_b_m}  that
\[
|\tilde S_{z}(\tau)|\le \frac{C}{|\tau|^m},
\]
which shows \eqref{eq_estim_symbol_S_0} for $j=0$ in the case $|\tau|> 1$.

To establish \eqref{eq_estim_symbol_S_0} for  $j=1,2,\dots$ in the case $|\tau|> 1$, we write
\begin{equation}
\label{eq_S_0_proof_2}
\begin{aligned}
d_\tau^{j}\tilde S_{z}(\tau)= \frac{ i}{ mz^{m-1}}\sum_{k=0}^{m/2-1} e^{2\pi k i/m} \bigg(
\int_{-\infty}^{0} \tilde\rho (|z|t) \rho(t)e^{-it\tau_k} (it)^{j} e^{it\tau}d t \\
 + \int_{0}^{+\infty} \tilde\rho (|z|t) \rho(t)e^{it\tau_k} (it)^{j} e^{it\tau}d t \bigg),
\end{aligned}
\end{equation}
and integrate by parts $(m+j)$ times in \eqref{eq_S_0_proof_2}. Due to the appearance of the terms $t^j$ in the integrands in \eqref{eq_S_0_proof_2}, no boundary terms arise when integrating by parts the first $j$ times.
Integrating by parts further, the contributions to the boundary terms that one has to consider would be similar to those in the case $j=0$, and therefore, we need only to discuss the integrals obtained after an integration by parts $m+j$ times  in \eqref{eq_S_0_proof_2}. The worst case scenario here occurs when  no derivatives fall on $\rho(t)$,  and the corresponding contribution to the integrals can be bounded by a constant times
\[
\bigg|\frac{1}{\tau^{m+j}}\int_{-1/|z|}^0 |z|^{l_1} (d_t^{l_1}\tilde\rho)(|z|t)\rho(t)(-i\tau_k)^{l_2}e^{-it\tau_k}t^{j-l_3}e^{it\tau}dt\bigg|\le C|z|^{m-1}\frac{1}{|\tau|^{m+j}}.
\]
Here $l_1+l_2+l_3=m+j$, $0\le l_3\le j$.
Together with \eqref{eq_S_0_proof_2} this implies \eqref{eq_estim_symbol_S_0}. The proof is complete.
\end{proof}

Combing Lemma \ref{lem_symbol} with  the fact that $Q\in \Psi^1_{\text{cl}}(M)$ is elliptic and self-adjoint,  we conclude from \cite[Theorem 4.3.1]{Sogge_book} that
 $\tilde S_{z}(Q)$ is a pseudodifferential operator of order $-m$, with the symbol seminorms uniformly bounded in $z\in \C$, $|z|\ge 1$.

Let  $\tilde S_{z}(Q)(x,y)\in \mathcal{D}'(M\times M)$ be the Schwartz kernel of the operator $\tilde S_{z}(Q)$. Then $\tilde S_{z}(Q)(x,y)$ is $C^\infty$ away from the diagonal $\{(x,x):x\in M\}$.
By \cite[Proposition 1, p. 241]{Stein_book}, since $n-m>0$, we have near the diagonal, in local coordinates,
\[
|\tilde S_{z}(Q)(x,y)|\le C|x-y|^{m-n},
\]
uniformly in $z\in \C$, $|z|\ge 1$.
An application of the Hardy-Littlewood-Sobolev inequality gives the estimate \eqref{eq_estim_S_0}.

Let us now prove the estimate \eqref{eq_estim_S_j}. By the Riesz--Thorin interpolation theorem,    \eqref{eq_estim_S_j} follows, if we show that
that there is a constant $C=C(\delta)$ so that for all $z\in \Xi_{\delta}\cap \{z\in \C:|z|\ge 1\}$, we have
\begin{equation}
\label{eq_estim_S_j_L_2}
\|S_{z,j}(Q)f\|_{L^{2}(M)}\le C |z|^{-m} 2^{j}\|f\|_{L^2(M)},\quad  j=0,1,\dots,
\end{equation}
and
\begin{equation}
\label{eq_estim_S_j_L_infty}
\|S_{z,j}(Q)f\|_{L^{\infty}(M)}\le C |z|^{n-m} 2^{-\frac{(n-1)}{2}j}\|f\|_{L^1(M)},\quad  j=0,1,\dots.
\end{equation}
Here the interpolation parameter $\theta=\frac{n-m}{n}$, and
\[
(|z|^{-m} 2^{j})^{\theta}(|z|^{n-m} 2^{-\frac{(n-1)}{2}j})^{1-\theta}=2^{j\frac{2n-m-nm}{2n}}.
\]
When proving  the estimate \eqref{eq_estim_S_j_L_2}, we use the identity $\|e^{itQ}f\|_{L^2(M)}=\|f\|_{L^2(M)}$, $t\in\R$, the fact that $\beta(2^{-j}|z|t)=0$ when $|t|\notin[2^{j-1}/|z|,2^{j+1}/|z|]$, and Minkowski's inequality,  to get
\[
\|S_{z,j}(Q)f\|_{L^{2}(M)}\le \frac{C}{|z|^{m-1}}\int_{|t|\in[2^{j-1}/|z|,2^{j+1}/|z|]}\|e^{itQ}f\|_{L^2(M)}dt\le \frac{C}{|z|^{m}} 2^j\|f\|_{L^2(M)},
\]
uniformly in $z$, which shows \eqref{eq_estim_S_j_L_2}.

Now we are left with proving \eqref{eq_estim_S_j_L_infty}.  Let us denote by $K_{z,j}(x,y)$ the Schwartz kernel of the operator $S_{z,j}(Q)$. The estimate \eqref{eq_estim_S_j_L_infty}  is implied by the estimate
\begin{equation}
\label{eq_kernel_desired}
|K_{z,j}(x,y)|\le C  |z|^{n-m} 2^{-\frac{(n-1)}{2}j},\quad x,y\in M,
\end{equation}
for all $z\in \Xi_{\delta}\cap \{z\in \C:|z|\ge 1\}$, uniformly in $z$.   By \eqref{eq_S_j}, we have
\begin{equation}
\label{eq_kernel_K_z_j}
K_{z,j}(x,y)=\frac{ i}{mz^{m-1}}\sum_{k=0}^{m/2-1} e^{2\pi k i/m} \int_{-\infty}^{+\infty} \beta (2^{-j}|z|t) \rho(t)e^{i|t|\tau_k} e^{it Q}(x,y)d t,
\end{equation}
where $ e^{it Q}(x,y)$ is the  Schwartz kernel of the half-wave operator $ e^{it Q}$. To proceed, we shall make use of
the  H\"ormander--Lax parametrix for the the half-wave operator $ e^{it Q}$, see \cite{Hormander_1968}, \cite[Theorem 4.1.2]{Sogge_book}.

\begin{lem}
\label{lem_Lax_parametrix}
Let  $Q\in \Psi^1_{\emph{\text{cl}}}(M)$ be elliptic and self-adjoint with respect to a positive $C^\infty$ density $d\mu$, and $q(x,\xi)$ be the principal symbol of $Q$. Then there is $\varepsilon>0$ small, depending on $M$ and $Q$, so that if $|t|<\varepsilon$,
\[
e^{itQ}=G(t)+R(t),
\]
where the remainder $R(t)$ has the kernel $R(t,x,y)\in C^\infty([-\varepsilon,\varepsilon]\times M\times M)$, and the kernel $G(t,x,y)$ is supported in a small neighborhood of the diagonal in $M\times M$, for $|t|<\varepsilon$.  Furthermore, suppose that local coordinates are chosen in a patch $\Omega\subset M$ so that $d\mu$ agrees with the Lebesque measure in the corresponding open subset of $\R^n$.  If $\omega\subset \Omega$ is relatively compact, $G(t,x,y)$ has the form,
\[
G(t,x,y)=(2\pi)^{-n}\int_{\R^n} e^{i [\varphi(x,y,\xi)+tq(y,\xi)]}g(t,x,y,\xi)d\xi
\]
when $(t,x,y)\in [-\varepsilon,\varepsilon]\times M\times \omega$. Here $g\in S^{0}_{1,0}$, i.e.
\[
|\p_\xi^\alpha\p_t^{\beta_1}\p_x^{\beta_2}\p_y^{\beta_3}g(t,x,y,\xi)|\le C_{\alpha,\beta_1,\beta_2,\beta_3}(1+|\xi|)^{-|\alpha|},
\]
for all multi-indices  $\alpha$, $\beta_1$, $\beta_2$, $\beta_3$,
and $g$ is supported in a small neighborhood of the diagonal in $\omega\times \omega$, and $\varphi$ is a real function which is homogeneous of degree one in $\xi$, $C^\infty$ for $\xi\ne 0$, and satisfies
\begin{equation}
\label{eq_phase_varphi}
\varphi(x,y,\xi)=\langle x-y,\xi \rangle+ \mathcal{O}_{S^1}(|x-y|^2|\xi|),
\end{equation}
i.e.
\[
|\p_\xi^{\alpha}(\varphi(x,y,\xi)-\langle x-y,\xi \rangle)|\le C_\alpha |x-y|^2|\xi|^{1-|\alpha|},
\]
for all multi-indices $\alpha$.
\end{lem}

In what follows, we shall make the choice of $\varepsilon$  in the definition \eqref{eq_def_rho_eps} of the function $\rho(t)$ so that  Lemma \ref{lem_Lax_parametrix} is applicable.

We assume that $2^{-j}|z|> 1$, as otherwise $S_{z,j}=0$, cf. \eqref{eq_S_j=0}.  Let us write
\[
K_{z,j}(x,y)=K_{z,j}^{(1)}(x,y)+ K_{z,j}^{(2)}(x,y),
\]
where
\begin{align*}
K_{z,j}^{(1)}(x,y)=
\frac{ i}{mz^{m-1}}\sum_{k=0}^{m/2-1} e^{2\pi k i/m} \int_{-\infty}^{+\infty} \beta (2^{-j}|z|t) \rho(t)e^{i|t|\tau_k} G(t,x,y)d t,\\
K_{z,j}^{(2)}(x,y)=
\frac{ i}{mz^{m-1}}\sum_{k=0}^{m/2-1} e^{2\pi k i/m} \int_{-\infty}^{+\infty} \beta (2^{-j}|z|t) \rho(t)e^{i|t|\tau_k} R(t,x,y)d t.
\end{align*}

Since $R(t,x,y)\in C^\infty([-\varepsilon,\varepsilon]\times M\times M)$, we have
\begin{equation}
\label{eq_kernel_smooth}
|K_{z,j}^{(2)}(x,y)|\le \frac{C}{|z|^{m-1}}\bigg|\int_{|t|\in [2^{j-1}/|z|,2^{j+1}/|z|]}dt\bigg|\le \frac{2^jC}{|z|^m}.
\end{equation}
As $2^{-j}|z|> 1$, the estimate \eqref{eq_kernel_smooth} is better than the desired bound \eqref{eq_kernel_desired} for $K_{z,j}$.

Let us now estimate $K_{z,j}^{(1)}$. Setting
\[
r=\frac{2^j}{|z|},\quad \frac{1}{|z|}\le r< 1,
\]
and assuming that the local coordinates are chosen as in Lemma \ref{lem_Lax_parametrix},
we write
\begin{equation}
\label{eq_K_z,j_phase}
\begin{aligned}
K_{z,j}^{(1)}(x,y)=&
\frac{ i}{mz^{m-1}}\sum_{k=0}^{m/2-1} e^{2\pi k i/m} \frac{1}{(2\pi)^n}\\
&\int_{\R^n}\int_{-\infty}^{+\infty} \beta (t/r) \rho(t)e^{i|t|\tau_k}
e^{i [\varphi(x,y,\xi)+tq(y,\xi)]}g(t,x,y,\xi)d t d\xi,
\end{aligned}
\end{equation}
for $(x,y)\in M\times \omega$.
We would like to replace $\varphi$ by the Euclidean phase function $\varphi_0=\langle x-y,\xi\rangle$.  In doing so, we shall follow \cite{Seeger_Sogge_1989} and notice that both $\varphi$ and $\varphi_0$ parametrize the trivial Lagrangian manifold $\{(x,\xi,x,\xi)\}$. This is due to the fact that when $(x,y)$ is in a neighborhood of the diagonal, we have
$\varphi'_\xi=0$ precisely when $x=y$, and then $\varphi'_x=-\varphi'_y=\xi$.  Following \cite{Seeger_Sogge_1989}, we can use the following result of \cite[Theorem 3.1.6]{Hormander_1971}.

\begin{lem}
\label{lem_change_coordinates}
Suppose that $\varphi$ is as in Lemma \ref{lem_Lax_parametrix}, i.e. $\varphi$ satisfies \eqref{eq_phase_varphi}. Then, for $(x,y)$ close to the diagonal, there is a $C^\infty$ for $\xi\ne 0$ homogeneous of degree one change of coordinates
\[
\eta=\kappa_{x,y}(\xi)
\]
so that
\[
\varphi(x,y,\kappa_{x,y}^{-1}(\eta))=\langle x-y,\eta\rangle.
\]
The transformation $\kappa_{x,y}$ depends smoothly on the parameters $x$, $y$, and in addition,
\begin{equation}
\label{eq_transform_identity}
\kappa_{x,y}=\emph{\text{Identity}}, \quad\text{when}\quad x=y.
\end{equation}
\end{lem}

Lemma \ref{lem_change_coordinates} implies that \eqref{eq_K_z,j_phase} can be rewritten as
 \begin{equation}
\label{eq_K_z,j_phase_2}
\begin{aligned}
K_{z,j}^{(1)}(x,y)=&
\frac{ i}{mz^{m-1}}\sum_{k=0}^{m/2-1} e^{2\pi k i/m} \frac{1}{(2\pi)^n}\\
&\int_{\R^n}\int_{-\infty}^{+\infty} \beta (t/r) \rho(t)e^{i|t|\tau_k}
e^{i [  \langle x-y,\eta\rangle +t\tilde q(x,y,\eta)]}\tilde g(t,x,y,\eta)d t d\eta,
\end{aligned}
\end{equation}
where
\[
\tilde g(t,x,y,\eta)= g(t,x,y,\kappa_{x,y}^{-1}(\eta))\bigg|\frac{D(\kappa_{x,y}^{-1})(\eta)}{D\eta}\bigg|,
\]
with $\frac{D(\kappa_{x,y}^{-1})(\eta)}{D\eta}$ being the Jacobian of the transformation $\kappa_{x,y}^{-1}$, has the same properties as $g$, in particular $\tilde g\in S^0_{1,0}$.
Also,
\[
\tilde q(x,y,\eta)=q(y,\kappa_{x,y}^{-1}(\eta))
\]
depends smoothly on $x$, $y$. Furthermore, since strict convexity is preserved under diffeomorphisms that are sufficiently close to the identity in the $C^\infty$ sense, the surface
\[
\tilde \Sigma_{x,y}=\{\eta\in \R^n: \tilde q(x,y,\eta)=1\}
\]
is  strictly convex.

Making the change of variables $t\mapsto t/r$ in \eqref{eq_K_z,j_phase_2}, we get
 \begin{equation}
\label{eq_K_z,j_phase_3}
\begin{aligned}
K_{z,j}^{(1)}(x,y)=&
\frac{ i r}{mz^{m-1}}\sum_{k=0}^{m/2-1} e^{2\pi k i/m} \frac{1}{(2\pi)^n}\\
&\int_{\R^n}\int_{-\infty}^{+\infty} \beta (t) \rho(rt)e^{i r|t|\tau_k}
e^{i  \langle x-y,\eta\rangle} e^{itr\tilde q(x,y,\eta)}\tilde g(rt,x,y,\eta)d t d\eta.
\end{aligned}
\end{equation}
As $q$ and $\kappa_{x,y}$ are homogeneous of degree one, we have
\[
r\tilde q(x,y,\eta)=q(x,y,r\kappa_{x,y}^{-1}(\eta))=\tilde q(x, y,r\eta).
\]
Making further  change of variables $\eta\mapsto r\eta$ in \eqref{eq_K_z,j_phase_3}, we obtain that
 \begin{equation}
\label{eq_K_z,j_phase_4}
\begin{aligned}
K_{z,j}^{(1)}(x,y)=&
\frac{ i r^{1-n}}{mz^{m-1}}\sum_{k=0}^{m/2-1} e^{2\pi k i/m} \frac{1}{(2\pi)^n}\\
&\int_{\R^n}\int_{-\infty}^{+\infty} \beta (t) \rho(rt)e^{i r|t|\tau_k}
e^{i  \langle \frac{x-y}{r},\eta\rangle} e^{it\tilde q(x, y,\eta)}\tilde g(rt,x,y,\eta/r)d t d\eta.
\end{aligned}
\end{equation}

As $\tilde q(x, y,\eta)$ is not smooth at $\eta=0$, it will be convenient to write
\begin{align*}
J_1(x,y,t,r)&=\int_{\R^n} e^{i [\langle \frac{x-y}{r},\eta\rangle + t\tilde q(x,y,\eta)]}\chi(\eta)\tilde g(rt,x,y,\eta/r) d\eta,\\
J_2(x,y,t,r)&=\int_{\R^n} e^{i [\langle \frac{x-y}{r},\eta\rangle + t\tilde q(x,y,\eta)]}(1-\chi(\eta))\tilde g(rt,x,y,\eta/r) d\eta,
\end{align*}
where   $\chi\in C^\infty_0(\R^n)$ and $\chi=1$ when $|\eta|\le 1$.  Here $|t|\in [1/2,2]$ and $0< r\le 1$.

As $\tilde g\in S^0_{1,0}$, we see that
\begin{equation}
\label{eq_J_1_est}
|J_1(x,y,t,r)|\le C,
\end{equation}
for all $x,y\in\omega$, $|x-y|$ small enough, uniformly in $r$.

Let us now estimate the absolute value of the oscillatory integral $J_2(x,y,t,r)$ when $|t|\in[1/2,2]$. To that end, consider
\[
\nabla_\eta [\langle \frac{x-y}{r},\eta\rangle + t\tilde q(x,y,\eta)], \quad |t|\in[1/2,2].
\]
As $\tilde q(x,y,\eta)$ is homogeneous of degree one in $\eta$, by the Euler homogeneity relation, we have
\[
\eta\cdot\nabla_\eta \tilde q(x,y,\eta)=\tilde q(x,y,\eta).
\]
This and the ellipticity of $\tilde q$ imply that $\nabla_\eta \tilde q(x,y,\eta)\ne 0$ for all $\eta\in \R^n\setminus\{0\}$. Thus, there is a constant $A>1/2$ such that $|\nabla_\eta \tilde q(x,y,\eta)|\ge A^{-1}$ for all $\eta\in\mathbb{S}^{n-1}$, and therefore, by the fact that $\nabla_\eta \tilde q$ is homogeneous of degree zero, we conclude that
\[
|\nabla_\eta \tilde q(x,y,\eta)|\ge A^{-1} \quad \text{for all}\quad  \eta\in \R^n\setminus\{0\}.
\]
On the other hand, since $\nabla_\eta \tilde q\in S^0_{1,0}$, for $|\eta|\ge 1$, we have
\[
|\nabla_\eta \tilde q(x,y,\eta)|\le A.
\]
Hence, for $|t|\in [1/2,2]$, if $x$, $y$ are such that
\begin{equation}
\label{eq_regim_1}
\frac{|x-y|}{r}\notin [ A^{-1}/4, 4 A],
\end{equation}
then
\begin{equation}
\label{eq_regim_1_grad}
|\nabla_\eta [\langle \frac{x-y}{r},\eta\rangle + t\tilde q(x,y,\eta)]|\ge A^{-1}/2.
\end{equation}

Assume first that \eqref{eq_regim_1} holds. Then we shall integrate by parts in the oscillatory integral $J_2$, see \cite[Lemma 1.2.1]{Hormander_1971}.  To that end,
setting
\[
\psi(t,x,y,\eta)=\langle \frac{x-y}{r},\eta\rangle + t\tilde q(x,y,\eta),
\]
we consider the operator
\[
L=\sum_{j=1}^n a_j \p_{\eta_j},\quad a_j=\frac{\p_{\eta_j}\psi}{i|\nabla_\eta\psi|^2}.
\]
We have $L^N(e^{i\psi(\eta)})=e^{i\psi(\eta)}$ for any $N\in \N$, and the transpose $L'$ of $L$ is given by
\begin{equation}
\label{eq_L_transpose}
L'=-\sum_{j=1}^n a_j \p_{\eta_j}-\div a, \quad a=(a_1,\dots,a_n).
\end{equation}
Hence, we get
\[
J_2(x,y,t,r)=\int_{\R^n} e^{i \psi(\eta)}(L')^N((1-\chi(\eta))\tilde g(rt,x,y,\eta/r)) d\eta.
\]

We observe that
\begin{equation}
\label{eq_symbol_g}
(1-\chi(\eta))\tilde g(rt,x,y,\eta/r)\in S^0_{1,0}
\end{equation}
uniformly in $0<r\le 1$. This follows from the facts that when $|\eta|\ge1$,
\[
|\p_\eta^\alpha\p_t^{\beta_1}\p_x^{\beta_2}\p_y^{\beta_3} \tilde g(rt,x,y,\eta/r)|\le \frac{r^{\beta_1}}{r^{|\alpha|}} C_{\alpha,\beta_1,\beta_2,\beta_3}(1+|\eta|/r)^{-|\alpha|}\le C_{\alpha,\beta_1,\beta_2,\beta_3}(1+|\eta|)^{-|\alpha|},
\]
for all  $\beta_1\in\N_0:=\N\cup\{0\}$ and all  $\alpha, \beta_2,\beta_3\in \N_0^n$, and
\[
|\p_\eta^\alpha \chi(\eta)|\le C_{\alpha,N}(1+|\eta|)^{-N},
\]
for all  $\alpha\in \N^n_0$ and all $N>0$.

Let us now show that
\begin{equation}
\label{eq_symbol_a_j}
a_j(\eta)\in S^0_{1,0},  \quad |\eta|\ge 1,
\end{equation}
uniformly in $r$, $x$, $y$ and $t$ satisfying \eqref{eq_regim_1}. Indeed, first using \eqref{eq_regim_1_grad}, we have
\begin{equation}
\label{eq_eta_der_4}
|a_j(\eta)|=\frac{|\p_{\eta_j}\psi|}{|\nabla_\eta\psi|^2}\le 2A.
\end{equation}
 Let  $\alpha\in \N^n$ be such that $|\alpha|\ge 1$. Then by Leibniz formula, we get
\begin{equation}
\label{eq_eta_der_1}
\p_\eta^\alpha a_j(\eta)=\sum_{\beta+\gamma=\alpha}c_{\beta,\gamma} \p_\eta^\beta(\p_{\eta_j}\psi)\p_\eta^\gamma\bigg(\frac{1}{|\nabla_\eta\psi|^2}\bigg),
\end{equation}
with constants $c_{\beta,\gamma}$.  Here
\[
\p_{\eta_j}\psi=\frac{x_j-y_j}{r}+t\p_{\eta_j}\tilde q(x, y,\eta),
\]
and hence, for $|\beta|\ge 1$, we have
\begin{equation}
\label{eq_eta_der}
|\p_\eta^\beta (\p_{\eta_j}\psi)|\le C_\beta (1+|\eta|)^{-|\beta|},
\end{equation}
uniformly in $r$. To estimate the absolute value of $\p_\eta^\gamma(1/|\nabla_\eta\psi|^2)$ for $|\gamma|\ge 1$, we shall use the Fa\`a di Bruno formula, see \cite[p. 94]{Zworski_book},
\begin{equation}
\label{eq_Faa_di_Bruno}
\p_\eta^\gamma\bigg(\frac{1}{b}\bigg)=\frac{1}{b}\sum_{\underset{|\gamma^j|\ge 1}{\underset{|\gamma|=|\gamma^1|+\dots+|\gamma^k|}{1 \le k\le |\gamma|}}}
C_{\gamma^1,\dots,\gamma^k}\prod_{j=1}^k\frac{\p_\eta^{\gamma^j} b}{b}.
\end{equation}
For $|\gamma^j|\ge 1$, using Leibniz formula and \eqref{eq_eta_der},
we have
\[
|\p_\eta^{\gamma^j} (|\nabla_\eta\psi|^2)|
\le C_{\gamma^{j}}|\nabla_\eta\psi|(1+|\eta|)^{-|\gamma^j|}.
\]
Therefore, \eqref{eq_Faa_di_Bruno} implies that for $\gamma\in \N_0^n$,
\begin{equation}
\label{eq_eta_der_2}
\bigg|\p_\eta^\gamma\bigg(\frac{1}{|\nabla_\eta\psi|^2}\bigg)\bigg|\le C_\gamma\frac{1}{|\nabla_\eta\psi|^{2}} (1+|\eta|)^{-|\gamma|}
\end{equation}
uniformly in $r$.
We conclude from \eqref{eq_eta_der_1} with the help of \eqref{eq_eta_der} and \eqref{eq_eta_der_2} that for all $a\in \N^n$, $|\alpha|\ge 1$,
\begin{equation}
\label{eq_eta_der_5}
|\p_\eta^\alpha a_j(\eta)|\le C_\alpha (1+|\eta|)^{-|\alpha|},
\end{equation}
uniformly in $r$. Hence, \eqref{eq_symbol_a_j} follows from  \eqref{eq_eta_der_4} and \eqref{eq_eta_der_5}.

Using \eqref{eq_eta_der_5}, we obtain that
\begin{equation}
\label{eq_symbol_div_a}
\div a\in S^{-1}_{1,0}, \quad |\eta|\ge 1,
\end{equation}
 uniformly in $r$, $x$, $y$ and $t$ satisfying \eqref{eq_regim_1}.
Thus, it follows from \eqref{eq_L_transpose} with the help of \eqref{eq_symbol_a_j}, \eqref{eq_symbol_div_a} and \eqref{eq_symbol_g} that
\[
(L')^N((1-\chi(\eta))\tilde g(rt,x,y,\eta/r))\in S^{-N}_{1,0}
\]
uniformly in $r$,  $x$, $y$ and $t$ satisfying \eqref{eq_regim_1}.

Hence,  choosing $N$ sufficiently large, we conclude that
\begin{equation}
\label{eq_J_2_first_case}
|J_2(x,y,t,r)|\le C.
\end{equation}
Therefore, it follows from  \eqref{eq_K_z,j_phase_4}, \eqref{eq_J_1_est} and \eqref{eq_J_2_first_case} that
\begin{equation}
\label{eq_K_1_first_case}
|K^{(1)}_{z,j}(x,y)|\le C\frac{r^{1-n}}{|z|^{m-1}}=2^{j(1-n)}|z|^{n-m},
\end{equation}
when $x,y$ are such that $\frac{|x-y|}{r}\notin [ A^{-1}/4, 4 A]$. The estimate \eqref{eq_K_1_first_case} is better than the desired estimate \eqref{eq_kernel_desired}.

Assume now that $\frac{|x-y|}{r}\in [ A^{-1}/4, 4 A]$ and let us estimate the absolute value of $K^{(1)}_{z,j}(x,y)$ in this case. As above, we only need to estimate the absolute value of
\begin{align*}
K_{z,j}^{(1,2)}(x,y)=&
\frac{ i r^{1-n}}{mz^{m-1}}\sum_{k=0}^{m/2-1} e^{2\pi k i/m} \frac{1}{(2\pi)^n}\int_{\R^n}\int_{-\infty}^{+\infty} \beta (t) \rho(rt)e^{i r|t|\tau_k} \\
&e^{i  \langle \frac{x-y}{r},\eta\rangle} e^{it\tilde q(x,y,\eta)}(1-\chi(\eta))\tilde g(rt,x,y,\eta/r)d t d\eta,
\end{align*}
where $\chi\in C^\infty_0(\R^n)$ is such that $\chi=1$ when $|\eta|\le 1$.
Using  \eqref{eq_residue}, we get
 \begin{equation}
\label{eq_K_z,j_phase_10}
\begin{aligned}
K_{z,j}^{(1,2)}(x,y)=&
\frac{ r^{1-n}}{(2\pi)^{n+1}}\int_{-\infty}^{+\infty}\int_{\R^n}\int_{-\infty}^{+\infty} \frac{e^{it(-r\tau+\tilde q(x,y,\eta))}}{\tau^m-z^m} d\tau \\
&\beta (t) \rho(rt)
e^{i  \langle \frac{x-y}{r},\eta\rangle} (1-\chi(\eta))\tilde g(rt,x,y,\eta/r)d\eta dt.
\end{aligned}
\end{equation}
Making the change of variables $\tau\mapsto -r\tau +\tilde q(x, y,\eta)$, we obtain that
\begin{equation}
\label{eq_K_z,j_phase_11}
\begin{aligned}
K_{z,j}^{(1,2)}(x,y)=
\frac{ r^{-n}}{(2\pi)^{n}}\int_{-\infty}^{+\infty}\int_{\R^n} \frac{h_r(\tau,x,y,\eta) e^{i  \langle \frac{x-y}{r},\eta\rangle} }{(\frac{\tilde q(x, y,\eta)-\tau}{r})^m-z^m} d\eta d\tau,
\end{aligned}
\end{equation}
where
\begin{equation}
\label{eq_form_h_r_100}
h_r(\tau,x,y,\eta)=\frac{1}{2\pi}\int_{-\infty}^{+\infty} e^{it \tau} \beta(t)\rho(rt) (1-\chi(\eta)) \tilde g(rt,x,y,\eta/r)dt
\end{equation}
is the inverse Fourier transform of the compactly supported smooth function $t\mapsto \beta(t)\rho(rt) (1-\chi(\eta)) \tilde g(rt,x,y,\eta/r)$.

We have
\begin{equation}
\label{eq_form_h_r}
|\p_\eta^\gamma h_r(\tau,x,y,\eta)  |\le C_{N,\gamma} (1+|\tau|)^{-N}(1+ |\eta|)^{-|\gamma|},
\end{equation}
uniformly in $r$, for all $N>0$ and $\gamma\in \N^n_0$. This can be seen by using \eqref{eq_symbol_g} in the case $|\tau|\le 1$, and by
integrating by parts $N$ times in \eqref{eq_form_h_r_100} and using \eqref{eq_symbol_g} in the case $|\tau|\ge 1$.

We write
\[
\bigg(\frac{\tilde q(x,y,\eta)-\tau}{r}\bigg)^m-z^m=\prod_{k=0}^{m-1}\bigg(\frac{\tilde q(x,y,\eta)- \tau}{r}-ze^{2\pi ki/m}\bigg),
\]
and using a partial fraction decomposition, we get
\[
\frac{1}{(\frac{\tilde q(x,y,\eta)-\tau}{r})^m-z^m}=\frac{r}{z^{m-1}}\sum_{k=0}^{m-1} \frac{A_k}{\tilde q(x,y,\eta)- \tau-rze^{2\pi ki/m}},
\]
where
\[
A_k=\bigg(\prod_{\underset{l\ne k}{l=0}}^{m-1}(e^{2\pi k i/m}-e^{2\pi l i/m})\bigg)^{-1}.
\]
Thus, it follows from \eqref{eq_K_z,j_phase_11} that
\begin{equation}
\label{eq_K_z,j_phase_12}
\begin{aligned}
K_{z,j}^{(1,2)}(x,y)=
\frac{ r^{1-n}}{(2\pi)^{n}z^{m-1}}\sum_{k=0}^{m-1} A_k  \int_{-\infty}^{+\infty}\int_{\R^n} \frac{h_r(\tau,x,y,\eta) e^{i  \langle \frac{x-y}{r},\eta\rangle} }{\tilde q(x,y,\eta)-(\tau +rze^{2\pi k i/m})} d\eta d\tau.
\end{aligned}
\end{equation}

Recalling that $\arg(z)\in (0,2\pi/m)$, we see that $\tau +rze^{2\pi k i/m}$ avoids the real axis, for $k=0,\dots, m-1$.
To proceed further, we shall need the following result, similar to \cite[Proposition 2.4]{Bourgain_Shao_Sogge_Yao}.

\begin{lem}
\label{lem_estimate_bourgain_sogge}
Let $n\ge 2$ and let  $h\in C^\infty(\R^n\setminus\{0\})$ satisfy  the Mihlin-type condition,
\begin{equation}
\label{eq_nom_h}
|\p_\xi^\alpha h(\xi)|\le H_{\alpha} |\xi|^{-|\alpha|}, \quad \xi\ne 0,\quad \alpha\in \N^n_0.
\end{equation}
Let $a \in C^\infty(\R^n\setminus\{0\})$ be homogeneous of degree one. Assume that $a(\xi)>0$ for all $\xi\in\R^n\setminus\{0\}$ and that  the cosphere $\Sigma=\{\xi\in \R^n:a(\xi)=1\}$ is strictly convex.
Then there is a constant $C>0$ such that for all $x\in \R^n$, $x\ne 0$, and all $w\in \C\setminus  [0,\infty)$, we have
\begin{equation}
\label{eq_lem_estimate_bourgain_sogge}
\bigg|\int_{\R^n}\frac{h(\xi) e^{i\langle x,\xi\rangle}}{a(\xi)-w}d\xi\bigg|\le C(|x|^{1-n}+(|w|/|x|)^{\frac{n-1}{2}}).
\end{equation}
\end{lem}

\begin{proof}
First notice that since $a \in C^\infty(\R^n\setminus\{0\})$ is homogeneous of degree one, we have
\[
|\p_\xi^\alpha a(\xi)|\le A_{\alpha} |\xi|^{1-|\alpha|}, \quad\quad \xi\ne 0,\quad \alpha\in \N^n_0.
\]

Let $b\in C^\infty(\R^n\setminus\{0\})$ be such that
\[
|\p_\xi^\alpha b(\xi)|\le B_{\alpha} |\xi|^{-1-|\alpha|}, \quad\xi\ne 0,\quad  \alpha\in \N^n_0.
\]
Then it follows from
\cite[p. 245]{Stein_book} that the  Fourier transform of $b(\xi)$ satisfies
\begin{equation}
\label{eq_lem_a_0}
\bigg| \int_{\R^n}  b(\xi)e^{-i\langle x,\xi\rangle} d\xi \bigg|\le C|x|^{1-n}, \quad x\ne 0.
\end{equation}

Assume first that $w$ is outside of a small but fixed conic neighborhood of the positive real axis $[0,\infty)$, i.e. $\arg w\in [\theta,2\pi -\theta]$ for some $\theta>0$ small but fixed, and  $|w|=1$.
Let us establish that
\[
b_w(\xi)=\frac{h(\xi)}{a(\xi)-w}\in C^\infty(\R^n\setminus\{0\}),
\]
satisfies
\begin{equation}
\label{eq_lem_a_3}
|\p_\xi^\alpha b_w(\xi)|\le B_{\alpha} |\xi|^{-1-|\alpha|}, \quad\xi\ne 0, \quad \alpha\in \N^n_0,
\end{equation}
uniformly in $w$.

To that end, let us show that
\begin{equation}
\label{eq_den_1}
|a(\xi)-w|\ge \frac{1}{C_\theta}(|\xi|+1),
\end{equation}
with a constant $C_\theta>0$ uniformly in $w$. When doing so, we notice there is a constant $\delta>0$ such that
$a(\xi)\ge \delta |\xi|$,
and then  \eqref{eq_den_1} follows for all $|\xi|$ large enough.
It remains to consider the case when $|\xi|$ is bounded.
Then if $\arg w\in [\theta, \pi-\theta]\cup [\pi+\theta,2\pi -\theta]$, we get
\[
|a(\xi)-w|\ge |\textrm{Im}(w)|\ge \frac{1}{C_\theta}.
\]
If $\arg w\in (\pi-\theta,\pi+\theta)$, we write $\arg w=\pi+\mathcal{O}(\theta)$. Then $w=-1-\mathcal{O}(\theta)$, and therefore,
 \[
 |a(\xi)-w|=|a(\xi)+1+\mathcal{O}(\theta)|\ge \frac{1}{2},
 \]
 for $\theta$ small enough. The bound \eqref{eq_den_1} follows.

By the Leibniz formula we write
\begin{equation}
\label{eq_den_5}
\p_\xi^\alpha (b_w(\xi))=\sum_{\beta+\gamma=\alpha} C_{\beta,\gamma}\p_\xi^\beta (h(\xi))\p_\xi^\gamma \bigg( \frac{1}{a(\xi)-w}\bigg),
\end{equation}
with constants $C_{\beta,\gamma}$. It follows from the Fa\`a di Bruno formula \eqref{eq_Faa_di_Bruno} and \eqref{eq_den_1} that for $|\gamma|\ge 0$,
\begin{equation}
\label{eq_den_6}
\bigg| \p_\xi^\gamma\bigg( \frac{1}{a(\xi)-w}\bigg)\bigg|\le C_{\gamma,\theta} |\xi|^{-1-|\gamma|}, \quad \xi\ne 0,
\end{equation}
uniformly in $w$.
Hence, we conclude from \eqref{eq_den_5}, with the help of \eqref{eq_nom_h}  and \eqref{eq_den_6}, that  \eqref{eq_lem_a_3} holds.

Thus, applying \eqref{eq_lem_a_0} for $b_w$, we obtain that
\begin{equation}
\label{eq_lem_a_4}
\bigg|\int_{\R^n}\frac{h(\xi) e^{i\langle x,\xi\rangle}}{a(\xi)-w}d\xi\bigg|\le C|x|^{1-n}, \quad x\ne 0,
\end{equation}
uniformly in $w\in \C$,  $\arg w\in [\theta,2\pi -\theta]$ with $\theta>0$ small but fixed, and $|w|=1$.

Assume now that $w\in \C$,  $\arg w\in [\theta,2\pi -\theta]$ with $\theta>0$ small but fixed,  and $|w|\ne 1$. Letting  $\tilde w=w/|w|$, we have
\[
\int_{\R^n}\frac{h(\xi) e^{i\langle x,\xi\rangle}}{a(\xi)-w}d\xi=\frac{1}{|w|}\int_{\R^n}\frac{h(\xi) e^{i\langle x,\xi\rangle}}{a(\xi/|w|)-\tilde w}d\xi=|w|^{n-1}
\int_{\R^n}\frac{h(|w| \xi) e^{i\langle |w| x,\xi\rangle}}{a(\xi)-\tilde w}d\xi.
\]
Since the dilate $h(|w|\xi)$ of $h(\xi)$ satisfies exactly the same bounds as in \eqref{eq_nom_h},  as above, we obtain the uniform estimate \eqref{eq_lem_a_4}, for all $w\in \C$,  $\arg w\in [\theta,2\pi -\theta]$ with $\theta>0$ small but fixed.

Assume now that  $w\in \C\setminus[0,\infty)$,  $\arg w\in (-\theta, \theta)$ with $\theta>0$  small but fixed, and $|w|=1$. Then   $w=1+\mathcal{O}(\theta)$, and therefore,
\[
|a(\xi)-w|=|a(\xi)-1-\mathcal{O}(\theta)|\ge \frac{1}{C},
\]
for $\xi\notin a^{-1} ([1/2,2])$, uniformly in $w$. Hence, letting $0\le \chi\in C_0^\infty((0,\infty))$ be such that $\chi(t)=1$ when $t\in [1/2,2]$ and $\supp(\chi)\subset [1/4, 4]$, by the above argument,  we conclude that
\[
b_w(\xi):=\frac{h(\xi) (1-\chi(a(\xi)))}{a(\xi)-w}
\]
satisfies the bound \eqref{eq_lem_a_3} uniformly in $w$. Therefore,
\[
\bigg|\int_{\R^n}\frac{h(\xi)  (1-\chi(a(\xi))) e^{i\langle x,\xi\rangle}}{a(\xi)-w}d\xi\bigg|\le C|x|^{1-n},
\]
uniformly in $w\in \C\setminus[0,\infty)$,  $\arg w\in (-\theta, \theta)$ with $\theta>0$  small but fixed, and $|w|=1$.

Let us now write,
\begin{equation}
\label{eq_lem_a_4_1}
I(x)=\int_{\R^n}\frac{h(\xi)  \chi(a(\xi)) e^{i\langle x,\xi\rangle}}{a(\xi)-w}d\xi=I_1(x)+I_2(x),
\end{equation}
where
\[
I_1(x):=\int_{\R^n}\frac{h(\xi)  \chi(a(\xi)) (a(\xi)-w_1) e^{i\langle x,\xi\rangle}}{(a(\xi)-w_1)^2+w_2^2}d\xi, I_2(x)=\int_{\R^n} \frac{i h(\xi)  \chi(a(\xi)) w_2 e^{i\langle x,\xi\rangle}}{(a(\xi)-w_1)^2+w_2^2}d\xi.
\]
Here  $w_1=\Re w=1+\mathcal{O}(\mu^2)$,  $w_2=\Im w=\mu + \mathcal{O}(\mu^2)$, and $\mu:=\arg w$, $|\mu|$ small.

Using the coarea formula in the integral $I_2(x)$, we get
\begin{equation}
\label{eq_lem_a_5}
\begin{aligned}
|I_2(x)|&\le C|w_2|\int_{a^{-1}([1/4,4])}\frac{d\xi}{(a(\xi)-w_1)^2+w_2^2}\\
&=C|w_2|\int_{1/4}^4 \int_{a(\xi)=E}\frac{dS_{E}}{|\nabla_\xi a(\xi)|}\frac{dE}{(E-w_1)^2+w_2^2},
\end{aligned}
\end{equation}
where $dS_E$ is the Lebesque measure on the surface $a(\xi)=E$.

Let us notice that by Euler homogeneity relations for    $a(\xi)=E$, we have
\[
|\nabla_\xi a(\xi)|\ge 1/C, \quad
\]
uniformly in $E\in [1/4,4]$.
Therefore,
\begin{equation}
\label{eq_lem_a_6_1}
|I_2(x)|\le C|w_2| \int_{1/4}^4 \frac{dE}{(E-w_1)^2+w_2^2}\le C|w_2| \int_{-\infty}^{+\infty} \frac{dE}{E^2+w_2^2}\le C,
\end{equation}
uniformly in $\mu$.

Appealing to the coarea formula in the integral $I_1(x)$,  we get
\begin{equation}
\label{eq_lem_a_7}
\begin{aligned}
I_1(x)&=\int_{a^{-1}([1/4,4])}\frac{h(\xi) \chi(a(\xi)) (a(\xi)-w_1) e^{i\langle x,\xi\rangle}}{(a(\xi)-w_1)^2+w_2^2}d\xi\\
&=\int_{1/4}^4 \frac{(E-w_1)}{(E-w_1)^2+w_2^2} J(E,x)dE,
\end{aligned}
\end{equation}
where
\[
J(E,x)=\chi(E)\int_{a(\xi)=E} \frac{h(\xi)e^{i\langle x,\xi\rangle}}{|\nabla_\xi a(\xi)|}dS_E=E^{n-1} \chi(E)\int_{a(\xi)=1} \frac{h(E\xi) e^{i\langle x,E\xi\rangle}}{|\nabla_\xi a(\xi)|}dS_1.
\]
We see that $J(E,x)$ is $C^\infty$ in $E$, $x$.
Making the change of variables $E\mapsto E-w_1$ in \eqref{eq_lem_a_7}, we get
\begin{align*}
I_1(x)&=\bigg(\int_{1/4-w_1}^0 + \int_0^{w_1-1/4}+ \int_{w_1-1/4}^{4-w_1}\bigg) \frac{E}{E^2+w_2^2} J(E+w_1,x)dE\\
&=\int_{0}^{w_1-1/4}\frac{E(J(E+w_1,x)-J(-E+w_1,x))}{E^2+w_2^2} dE \\
&+ \int_{w_1-1/4}^{4-w_1} \frac{E}{E^2+w_2^2} J(E+w_1,x)dE.
\end{align*}
As $f(E)=J(E+w_1,x)-J(-E+w_1,x)$ is $C^\infty$  in  $E$, $w_1$, and $x$,  and $f(0)=0$, it follows that  $f(E)=Eg(E)$ with a function $g$ which is $C^\infty$   in  $E$, $w_1$, and $x$. Hence, recalling that $w_1=1+\mathcal{O}(\mu^2)$, for $|x|\le 1$,   we get
\begin{equation}
\label{eq_lem_a_8}
|I_1(x)|\le C\int_0^2 \frac{E^2}{E^2+w_2^2}dE +  C\int_{1/4}^4\frac{1}{E} dE\le C,
\end{equation}
uniformly in $\mu$ with $0<|\mu|\le \theta$, where $\theta$ is  sufficiently small.

We conclude from \eqref{eq_lem_a_4_1}, \eqref{eq_lem_a_6_1} and \eqref{eq_lem_a_8} that
\[
|I(x)|\le C,
\]
for  $|x|\le 1$, uniformly in $\mu$ with $0<|\mu|\le \theta$, where $\theta$ is  sufficiently small.

Let us now show that when $|x|\ge 1$, we get
\begin{equation}
\label{eq_lem_a_9}
|I(x)|\le C |x|^{-\frac{(n-1)}{2}},
\end{equation}
uniformly in $\mu$.   First using the coarea formula in  \eqref{eq_lem_a_4_1}, we get
\begin{align*}
I(x)&=\int_{1/4}^4 \int_{a(\xi)=E} \frac{h(\xi)\chi(E) e^{i\langle x,\xi\rangle}}{(E-w)}\frac{dS_E}{|\nabla_\xi a(\xi)|}dE\\
&=
\int_{1/4}^4 \frac{E^{n-1}\chi(E)}{E-w}  \int_{a(\xi)=1} \frac{h(E\xi)}{|\nabla_\xi a(\xi)|} e^{i\langle Ex,\xi\rangle}dS_1 dE.
\end{align*}

To proceed recall that $a(\xi)$ is  homogeneous of degree one, $C^\infty$ for $\xi\ne 0$, and $a(\xi)>0$ on $\R^n\setminus\{0\}$. Then $\nabla_\xi a\ne 0$ along the cosphere $\Sigma=\{\xi\in \R^n: a(\xi)=1\}$, which is therefore is a $C^\infty$ compact hypersurface.  Furthermore, $\Sigma$ is homeomorphic to the sphere $\mathbb{S}^{n-1}$ via  the homeomorphism  $\mathbb{S}^{n-1}\to \Sigma$, $\omega\mapsto \omega/a(\omega)$.  Hence, $\Sigma$ is connected.
The assumption that the Gaussian curvature of $\Sigma$ never vanishes implies that the Gauss map is a diffeomorphism from $\Sigma$ to $\mathbb{S}^{n-1}$.
Thus,  given $x\in \R^n\setminus\{0\}$,  there are exactly two points $\xi_1(x), \xi_2(x)\in \Sigma$ with normal $x$. Since $\xi_1(x)$, $\xi_2(x)$, are homogeneous of degree zero and smooth in $\R^n\setminus\{0\}$, it follows that  the functions $\langle x,\xi_1(x) \rangle$,  $\langle x,\xi_2(x) \rangle$ are also smooth for $x\ne 0$ and homogeneous of degree one.

We shall need the following result concerning  the inverse Fourier transform of a smooth measure carried by the cosphere $\Sigma$,  which is an application of the stationary phase theorem, see \cite[Theorem 1.2.1, p. 50]{Sogge_book} and \cite[p. 68]{Sogge_book}.

\begin{lem} \label{lem_stationary_phase}
Let $d\sigma(\xi)=\beta(\xi)dS(\xi)$ with $\beta\in C^\infty(\Sigma)$ and $dS$ is the surface measure on $\Sigma$. Then
under the above assumptions, the inverse Fourier transform of the measure $d\sigma$  satisfies
\[
(2\pi)^{-n}\int_{\Sigma} e^{i\langle x,\xi \rangle} d\sigma(\xi)=\frac{b_1(x) e^{i\langle x,\xi_1(x) \rangle}}{|x|^{(n-1)/2}}+ \frac{b_2(x) e^{i\langle x,\xi_2(x) \rangle}}{|x|^{(n-1)/2}}, \quad |x|\ge 1,
\]
where the functions $b_j$ are such that
\[
|\p_x^\alpha b_j(x)|\le C_\alpha |x|^{-|\alpha|}, \quad |x|\ge 1, \quad \alpha\in \N^n_0.
\]
\end{lem}

As $\xi_j(x)$ is homogeneous of degree zero, by Lemma \ref{lem_stationary_phase}, for $|x|\ge 1$, we get
\[
I(x)=(2\pi)^n |x|^{-\frac{(n-1)}{2}}\sum_{j=1}^2 \int_{1/4}^4 \frac{E^{(n-1)/2}\chi(E) b_j(x,E)}{E-w} e^{iE\langle x,\xi_j(x)\rangle}dE,
\]
with some functions $b_j\in C^\infty$ for $|x|\ge 1$ and $E\in [1/4,4]$, and
\begin{equation}
\label{eq_lem_a_10_0}
|\p_E^l \p_x^\alpha b_j(x,E)|\le C_{l,\alpha} |x|^{-|\alpha|}, \quad |x|\ge 1,\quad E\in [1/4,4],  \quad l\in \N_0, \quad \alpha\in \N^n_0.
\end{equation}
The estimate \eqref{eq_lem_a_9} would follow if we could show that
\begin{equation}
\label{eq_lem_a_10}
\bigg| \int_{1/4}^4 \frac{E^{(n-1)/2}\chi(E) b_j(x,E)}{E-w} e^{iE\langle x,\xi_j(x)\rangle}dE \bigg|\le C,
\end{equation}
uniformly in $\mu$, $0<|\mu|\le \theta\ll 1$.
To show \eqref{eq_lem_a_10}, we let
\[
f(E,x)=E^{(n-1)/2}\chi(E) b_j(x,E),\quad \varphi(x)=\langle x,\xi_j(x)\rangle.
\]
For $|x|\ge 1$, the function $f(\cdot,x)$ is $C^\infty$ with compact support in $E\in [1/4,4]$, and \eqref{eq_lem_a_10_0} yields that
\begin{equation}
\label{eq_lem_a_11}
|\p_E^l f(E,x)|\le C_l.
\end{equation}
We write
\begin{align*}
J(x)=&\int_{1/4}^4 \frac{f(E,x) e^{iE\varphi(x)}}{E-w}dE =\frac{1}{2\pi}\int_{-\infty}^{+\infty}\hat f(t,x)\int_{-\infty}^{+\infty} \frac{e^{iE(t+\varphi(x))}}{E-w_1-iw_2}dE dt\\
&=-\frac{1}{2\pi i} \int_{-\infty}^{+\infty}\hat f(t,x)e^{i w_1(t+\varphi(x))}\int_{-\infty}^{+\infty}\frac{e^{-i\tau(t+\varphi(x))}}{w_2-i\tau}d\tau dt,
\end{align*}
where $\hat f(t,x)$ is the Fourier transform of $f(E,x)$.  We shall use the following fact:
 for all $\alpha\in\R$, $\alpha\ne 0$,
\[
\frac{1}{2\pi}\int_{-\infty}^{+\infty} \frac{e^{-i\tau t}}{\alpha-i\tau} d\tau=\textrm{sgn} \alpha H(\alpha t) e^{-\alpha t},
\]
where $H(t)$ is the Heaviside function which equals one for $t\ge 0$ and zero for $t<0$, see \cite[Lemma 2.1]{Bourgain_Shao_Sogge_Yao}.   As $w_2\ne 0$, we get
\begin{align*}
J(x)= \int_{-\infty}^{+\infty}\hat f(t,x) i e^{iw_1(t+\varphi(x))}\textrm{sgn}(w_2) H(w_2(t+\varphi(x)))e^{-w_2(t+\varphi(x))}dt,
\end{align*}
and therefore, using that $f$ has compact support in $E$ and  \eqref{eq_lem_a_11}, we obtain that
\begin{align*}
|J(x)|&\le C\int_{-\infty}^{+\infty} |\hat f(t,x)|dt\le C\|(1+t^2)\hat f(t,x)\|_{L^\infty_t}\\
&\le C(\|f(E,x)\|_{L^1_E}+\|\p_E^2 f(E,x)\|_{L^1_E})\le C,
\end{align*}
uniformly in $w$. This establishes \eqref{eq_lem_a_10}, and hence, \eqref{eq_lem_a_9}.  Thus, for $w\in \C\setminus [0,\infty)$, $\textrm{arg} w\in  (-\theta,\theta)$, $\theta>0$ small but fixed, and $|w|=1$,   we get
\begin{equation}
\label{eq_lem_estimate_bourgain_sogge_2}
\bigg|\int_{\R^n}\frac{h(\xi) e^{i\langle x,\xi\rangle}}{a(\xi)-w}d\xi\bigg|\le C(|x|^{1-n}+|x|^{-\frac{(n-1)}{2}}), \quad x\ne 0,
\end{equation}
uniformly in $w$.  In the case when  $w\in \C\setminus [0,\infty)$, $\textrm{arg} w\in  (-\theta,\theta)$, $\theta>0$ small but fixed, and $|w|\ne 1$,  the estimate \eqref{eq_lem_estimate_bourgain_sogge} follows from \eqref{eq_lem_estimate_bourgain_sogge_2} by a change of scale. The proof of Lemma \ref{lem_estimate_bourgain_sogge} is complete.
\end{proof}

Now using Lemma \ref{lem_estimate_bourgain_sogge}, the estimate  \eqref{eq_form_h_r}, and the fact that $\frac{|x-y|}{r}\in [A^{-1}/4,4A]$, we obtain that
\begin{equation}
\label{eq_loc_by_lem_1}
\begin{aligned}
\bigg|\int_{\R^n} \frac{h_r(\tau,x,y,\eta) e^{i  \langle \frac{x-y}{r},\eta\rangle} }{\tilde q(x,y,\eta)-(\tau +rze^{2\pi k i/m})} d\eta\bigg|\le C_N(1+|\tau|)^{-N}(1+|\tau|+r|z|)^{\frac{n-1}{2}},
\end{aligned}
\end{equation}
for $k=0,1,\dots, m-1$ and $N>0$.
It follows from \eqref{eq_K_z,j_phase_12} and  \eqref{eq_loc_by_lem_1}  that for $N>0$ sufficiently large,
\begin{align*}
|K_{z,j}^{(1,2)}(x,y)|&\le C
\frac{ r^{1-n}}{|z|^{m-1}} \int_{-\infty}^{+\infty} (1+|\tau|)^{-N+\frac{n-1}{2}}(1+r|z|)^{\frac{n-1}{2}} d\tau\\
&\le Cr^{-\frac{(n-1)}{2}}|z|^{\frac{n+1-2m}{2}}.
\end{align*}
Here we have used that $r|z|\ge 1$.
Recalling that $r=2^j/|z|$, the above estimate completes the proof of the estimate \eqref{eq_kernel_desired}, and therefore, the estimates \eqref{eq_estim_S_j_L_infty} and \eqref{eq_estim_S_j}.
As $\sum_{j=0}^\infty 2^{j\frac{2n-m-nm}{2n}}=1/(1-2^{\frac{2n-m-nm}{2n}})$, we have obtained the \eqref{eq_resolvent_with_b_loc} for the local operator.

\subsection{Uniform estimate for the non-local operator in the case of unbounded $|z|$}
Let $\tau\in\R$ and consider the multipliers
\begin{equation}
\label{eq_rest_r}
r_z(\tau)=m_z(\tau)-m_z^{\textrm{loc}}(\tau)=\frac{i}{m z^{m-1}}\sum_{k=0}^{m/2-1} e^{2\pi k i/m}\int_{-\infty}^{+\infty} (1-\rho(t)) e^{i |t| \tau_k} e^{it\tau}dt,
\end{equation}
for all $z\in \Xi_\delta\cap\{z\in \C:|z|\ge 1\}$.

In order to prove \eqref{eq_resolvent_est}, we are left with establishing that
\begin{equation}
\label{eq_sec_2_3_0}
\|r_z(Q)f\|_{L^{\frac{2n}{n-m}}(M)}\le C\|f\|_{L^{\frac{2n}{n+m}}(M)},
\end{equation}
for all $z\in \Xi_\delta\cap\{z\in \C:|z|\ge 1\}$, uniformly in $z$.

Let us first show that $r_z(\tau)$ is bounded for all $z\in \Xi_\delta\cap\{z\in \C:|z|\ge 1\}$, uniformly in $z$. Indeed, we have
\begin{equation}
\label{eq_sec_2_3_1}
\begin{aligned}
|r_z(\tau)|\le \frac{C}{|z|^{m-1}} \sum_{k=0}^{m/2-1}\bigg(\int_{-\infty}^{-\varepsilon/2} e^{t\textrm{Im} \tau_k}dt + \int_{\varepsilon/2}^{+\infty} e^{-t\textrm{Im}\tau_k}dt\bigg)\le C \sum_{k=0}^{m/2-1}\frac{1}{\textrm{Im} \tau_k}.
\end{aligned}
\end{equation}
Recall that $\tau_k=ze^{2\pi ki/m}$, and therefore, $0<\arg(\tau_k)<\pi$, $k=0,\dots, m/2-1$. If now $0<\arg(\tau_k)\le \pi/2$, then
\[
\frac{\textrm{Im}\tau_k}{|z|}=\sin(\arg(\tau_k))\ge \sin(\arg(z)),
\]
and thus, using the fact that $z\in \Xi_\delta$, we get
\begin{equation}
\label{eq_sec_2_3_2}
\textrm{Im}\tau_k\ge \textrm{Im} z\ge \delta.
\end{equation}
If $\pi/2<\arg(\tau_k)< \pi$, then
\[
\frac{\textrm{Im}\tau_k}{|z|}=\sin(\pi- \arg(\tau_k))\ge \sin(\pi-\arg(\tau_{m/2-1}))=-\sin(\arg(z)-2\pi/m),
\]
and therefore,
\begin{equation}
\label{eq_sec_2_3_3}
\textrm{Im}\tau_k\ge -\textrm{Im} (ze^{-2\pi i/m})\ge \delta.
\end{equation}
Hence, it follows from \eqref{eq_sec_2_3_1}, \eqref{eq_sec_2_3_2} and \eqref{eq_sec_2_3_3} that
\begin{equation}
\label{eq_sec_2_3_4}
|r_z(\tau)|\le C\delta^{-1},
\end{equation}
for all $z\in \Xi_\delta\cap\{z\in \C:|z|\ge 1\}$, uniformly in $z$.

To obtain the decay of $r_z(\tau)$, let us integrate by parts $N$ times, $N=1,2,\dots$, in \eqref{eq_rest_r}.  We have
\begin{align*}
r_z(\tau)= \frac{i}{m z^{m-1}}\sum_{k=0}^{m/2-1} e^{2\pi k i/m}\bigg(& \frac{(-1)^N}{i^N(-\tau_k+\tau)^N}\int_{-\infty}^{0}(-\p_t^N \rho(t)) e^{it(-\tau_k+\tau)}dt\\
&+\frac{(-1)^N}{i^N(\tau_k+\tau)^N}\int_{0}^{+\infty}(-\p_t^N \rho(t)) e^{it(\tau_k+\tau)}dt\bigg).
\end{align*}
Notice that all the boundary terms disappear when integrating by parts due to the presence of the term $(1-\rho(t))$ in \eqref{eq_rest_r} and the fact that $\textrm{Im}\tau_k>0$. As
\begin{align*}
|\pm \tau_k+\tau|=\sqrt{|\pm \Re \tau_k+\tau|^2+|\Im \tau_k|^2}&\ge \sqrt{|\pm \Re \tau_k+\tau|^2+\delta^2}\\
&\ge \frac{\delta}{\sqrt{2}}(1+|\pm \Re \tau_k+\tau|),
\end{align*}
where $\delta<1$, we obtain that
\[
|r_z(\tau)|\le \frac{C}{|z|^{m-1}}\sum_{k=0}^{m/2-1}((1+|-\Re\tau_k+\tau|)^{-N}+ (1+|\Re\tau_k+\tau|)^{-N}),
\]
uniformly in $z$.  Thus, for $\tau\ge 0$, we get
\begin{equation}
\label{eq_sec_2_3_5}
|r_z(\tau)|\le \frac{C}{|z|^{m-1}}\bigg(\sum_{\underset{\textrm{Re} \tau_k\ge 0}{k=0,\dots, m/2-1}} (1+|-\Re\tau_k+\tau|)^{-N} +
\sum_{\underset{\textrm{Re} \tau_k< 0}{k=0,\dots, m/2-1}}
(1+|\Re\tau_k+\tau|)^{-N} \bigg)
\end{equation}

We have
\begin{equation}
\label{eq_sec_2_3_6}
r_z(Q)f=\sum_{j=1}^\infty r_z(\mu_j)E_jf=\sum_{l=1}^\infty r_z^l(Q)f, \quad f\in C^\infty(M),
\end{equation}
where
\[
r_z^l(Q)f=\sum_{\mu_j\in [l-1,l)}r_z(\mu_j) E_j f, \quad l=1,2,\dots.
\]
Using  Lemma \ref{lem_truncated_eq} and \eqref{eq_sec_2_3_5} with $N=m+1$, we obtain that
\begin{equation}
\label{eq_sec_2_3_7}
\begin{aligned}
\|r_z^l(Q)&f\|_{L^{\frac{2n}{n-m}}(M)}\le C l^{m-1}(\sup_{\tau\in [l-1,l)}|r_z(\tau)|)\|f\|_{L^{\frac{2n}{n+m}(M)}} \le  \frac{Cl^{m-1}}{|z|^{m-1}}  \\
&\bigg(\sum_{\underset{\textrm{Re} \tau_k\ge 0}{k=0}}^{ m/2-1} \frac{1}{(1+|-\Re\tau_k+l |)^{m+1}}
+
\sum_{\underset{\textrm{Re} \tau_k< 0}{k=0}}^{ m/2-1}
\frac{1}{(1+|\Re\tau_k+l |)^{m+1}} \bigg) \|f\|_{L^{\frac{2n}{n+m}}(M)}.
\end{aligned}
\end{equation}
Here we have used the fact that for $l-1\le \tau\le l$, we have
\[
|\pm\Re\tau_k+l |\le |\pm \Re \tau_k+\tau|+|l-\tau|\le  |\pm \Re \tau_k+\tau| +1.
\]

Hence, \eqref{eq_sec_2_3_0} would follow from \eqref{eq_sec_2_3_6} and \eqref{eq_sec_2_3_7}, if we could show that
\begin{equation}
\label{eq_sec_2_3_8}
\Sigma:=\frac{1}{|z|^{m-1}}  \sum_{l=1}^\infty \frac{l^{m-1}}{(1+|-a+l |)^{m+1}} \le  C, \quad a=|\Re \tau_k|,
\end{equation}
with some constant $C>0$ uniform in $z\in \C$, $|z|\ge 1$.

Let us now show \eqref{eq_sec_2_3_8}. Assume first that $a\le 1$. Then
\begin{align*}
\Sigma=\frac{1}{|z|^{m-1}}  \sum_{l=1}^\infty \frac{l^{m-1}}{(1-a+l )^{m+1}} \le \frac{1}{|z|^{m-1}}  \sum_{l=1}^\infty \frac{1}{l^2}\le C,
\end{align*}
with a constant $C>0$ uniform in $z\in \C$, $|z|\ge 1$. Consider now the case $a>1$. Then denoting $[a]$ the integer part of $a$, we write
\[
\Sigma=\Sigma_1+\Sigma_2 + \Sigma_3,
\]
where
\begin{align*}
\Sigma_1&:=\frac{1}{|z|^{m-1}}  \sum_{l\le [a]-1} \frac{l^{m-1}}{(1+a-l )^{m+1}},\\
\Sigma_2&:= \frac{1}{|z|^{m-1}} \bigg( \frac{[a]^{m-1}}{(1+|-a+[a] |)^{m+1}} + \frac{([a]+1)^{m-1}}{(1+|-a+[a]+1 |)^{m+1}} \bigg),\\
\Sigma_3&:=\frac{1}{|z|^{m-1}}  \sum_{l\ge [a]+2} \frac{l^{m-1}}{(1-a+l )^{m+1}}.
\end{align*}
Using the fact that $a\le |z|$, we see that $\Sigma_2 \le C$, uniformly in $z\in \C$, $|z|\ge 1$.

We shall next estimate $\Sigma_3$. As the function $t^{m-1}/(1-a+t)^{m+1}$ is decreasing for $t>0$, we get
\begin{align*}
\Sigma_3&\le \frac{1}{|z|^{m-1}} \int_{[a]+1}^{+\infty} \frac{t^{m-1}}{(1-a+t)^{m+1}}dt=\frac{1}{|z|^{m-1}} \int_{2+[a]-a}^{+\infty}\frac{(t+a-1)^{m-1}}{t^{m+1}}dt\\
&\le \frac{C_m}{|z|^{m-1}}\bigg(\int_1^{+\infty} \frac{dt}{t^2} + (a-1)^{m-1}\int_1^{+\infty} \frac{dt}{t^{m+1}}\bigg)\le C,
\end{align*}
uniformly in $z\in \C$, $|z|\ge 1$.

Let us now estimate $\Sigma_1$. Since the function $t^{m-1}/(1+a-t)^{m+1}$ is increasing for $t>0$, we obtain that
\begin{align*}
\Sigma_1\le \frac{1}{|z|^{m-1}}  \int_1^{[a]} \frac{t^{m-1}}{(1+a-t )^{m+1}}dt\le \frac{1}{|z|^{m-1}}  \int_{1+a-[a]}^{a} \frac{|1+a-t|^{m-1}}{t^{m+1}}dt\\
\le \frac{C_m}{|z|^{m-1}}\bigg( (1+a)^{m-1}\int_1^{+\infty} \frac{dt}{t^{m+1}} + \int_1^{+\infty} \frac{dt}{t^2}\bigg)\le C,
\end{align*}
uniformly in $z\in \C$, $|z|\ge 1$.
This completes the proof of \eqref{eq_sec_2_3_8} and hence, of Theorem \ref{thm_main}.

Finally let us remark that the a priori estimate \eqref{eq_resolvent_est} implies the following simple result concerning the $L^2$ resolvent of $P$, $(P-\zeta)^{-1}$.
\begin{prop}
\label{eq_prop_resolvent_on_L_p}
Let $\zeta\in\C\setminus[0,\infty)$. Then the resolvent $(P-\zeta)^{-1}$  is a bounded operator:  $L^{\frac{2n}{n+m}}(M)\to L^{\frac{2n}{n-m}}(M)$.
\end{prop}

\begin{proof}
Let $\zeta\notin\{\lambda_1,\lambda_2,\dots\}$  so that  $(P-\zeta)^{-1}:L^2(M)\to L^2(M)$ is  bounded.    By elliptic regularity, we have $(P-\zeta)^{-1}C^\infty(M)\subset C^\infty(M)$, and therefore, the linear continuous operator  $P-\zeta:C^\infty(M)\to C^\infty(M)$ is bijective. By the open mapping theorem,  $(P-\zeta)^{-1}: C^\infty(M)\to C^\infty(M)$ is continuous.

We have next the linear continuous map $P-\zeta:\mathcal{D}'(M)\to \mathcal{D}'(M)$ given by
\[
\langle (P-\zeta)u, \varphi\rangle =\langle u, \overline{(P-\overline{\zeta})\overline{\varphi}}\rangle , \quad \varphi\in C^\infty(M),
\]
which is bijective, with continuous inverse $(P-\zeta)^{-1}:\mathcal{D}'(M)\to \mathcal{D}'(M)$.

By Remark \ref{rem_non_uniform_est}, when $\zeta\in \C\setminus [0,\infty)$, we have the following a priori estimate
\[
\|u\|_{L^{\frac{2n}{n-m}}(M)}\le C_\zeta \|(P-\zeta)u\|_{L^{\frac{2n}{n+m}}(M)},
\]
for all $u\in C^\infty(M)$.  Thus, for any $f\in C^\infty(M)$, we get
\begin{equation}
\label{eq_operat_dist_2}
\| (P-\zeta)^{-1}f\|_{L^{\frac{2n}{n-m}}(M)}\le C_\zeta \|f\|_{L^{\frac{2n}{n+m}}(M)}.
\end{equation}
Now let $f\in L^{\frac{2n}{n+m}}(M)$. Then there is a sequence $f_j\in C^\infty(M)$, converging to $f$ in $L^{\frac{2n}{n+m}}(M)$ as $j\to \infty$. It follows from \eqref{eq_operat_dist_2} that $(P-\zeta)^{-1}f_j$
is a Cauchy sequence in $L^{\frac{2n}{n-m}}(M)$, and therefore, it converges in $L^{\frac{2n}{n-m}}(M)$.  As $(P-\zeta)^{-1}:\mathcal{D}'(M)\to \mathcal{D}'(M)$ continuous, we have $(P-\zeta)^{-1}f\in L^{\frac{2n}{n-m}}(M)$ and  $(P-\zeta)^{-1}f_j$ converges to $(P-\zeta)^{-1}f$ in $L^{\frac{2n}{n-m}}(M)$ as $j\to \infty$.
Hence, \eqref{eq_operat_dist_2} is valid for any $f\in L^{\frac{2n}{n+m}}(M)$, which shows the claim of Proposition \ref{eq_prop_resolvent_on_L_p}.
\end{proof}

\section{Saturation of the resolvent estimates. Proof of Theorem \ref{thm_main_2}}

\label{sec_proof_thm2}

We shall need the following Bernstein type inequality, established in  \cite[Lemma 3.1]{Bourgain_Shao_Sogge_Yao}.

\begin{lem}
\label{lem_Bernstein_type}
 Let $\beta\in C_0^\infty(\R)$ be such that $0\notin \supp(\beta)$. Then if $1\le q\le r\le \infty$, there is a constant $C=C(r,q)$ so that
\[
\|\beta(Q/\alpha) f\|_{L^r(M)}\le C\alpha^{n(\frac{1}{q}-\frac{1}{r})}\|f\|_{L^q(M)}, \quad \alpha \ge 1.
\]

\end{lem}

In Theorem \ref{thm_main} we obtained the uniform estimate \eqref{eq_resolvent_est} for all $z$ in the sector $\Xi$ of the complex plane such that $\dist(\p\Xi, z)\ge \delta$ for some $\delta>0$.
The next result shows that removing the eigenvalues of the operator $Q=P^{1/m}$ in some interval $[\alpha-1, \alpha+ 1]$ allows us to obtain the uniform estimate \eqref{eq_resolvent_est}  for all $z\in \Xi$ with $\Re z=\alpha\gg 1$ or $\Re (ze^{-2\pi i/m})=\alpha\gg 1$.

\begin{lem}
\label{lem_satur_1}
Let
\[
\chi_{[\alpha-1, \alpha+1)} f=\sum_{\mu_j\in [\alpha-1, \alpha+1)} E_j f.
\]
Then we have the uniform estimate:
\begin{equation}
\label{eq_4_1_1}
\|(I-\chi_{[\alpha-1, \alpha+1)})\circ (P-z^m)^{-1} f\|_{L^{\frac{2n}{n-m}}(M)}\le C\|f\|_{L^{\frac{2n}{n+m}}(M)},
\end{equation}
with $z\in \Xi$, $\emph{\textrm{Re }}z=\alpha\gg 1$, and $0<\emph{\textrm{Im }}z\le 1$, and the uniform estimate:
\begin{equation}
\label{eq_4_1_1_another}
\|(I-\chi_{[\alpha-1, \alpha+1)})\circ (P-z^m)^{-1} f\|_{L^{\frac{2n}{n-m}}(M)}\le C\|f\|_{L^{\frac{2n}{n+m}}(M)},
\end{equation}
with $z\in \Xi$, $\emph{\textrm{Re }}(ze^{-2\pi i/m})=\alpha\gg 1$, and $0<-\emph{\textrm{Im }}(z e^{-2\pi i /m})\le 1$.

\end{lem}

\begin{proof}
Let us start by proving \eqref{eq_4_1_1}.  Let $z\in \Xi$, $\textrm{Re }z=\alpha\gg 1$, and assume first that $\delta\le \textrm{Im }z= \beta\le 1$ for some $\delta>0$.
We write
\[
\chi_{[\alpha-1, \alpha+1)}\circ (P-z^m)^{-1} f=\sum_{\mu_j\in [\alpha-1, \alpha+1)} (\mu_j^m-z^m)^{-1}E_jf.
\]
By \eqref{eq_truncated}, we get
\begin{equation}
\label{eq_4_1_2}
\|\chi_{[\alpha-1, \alpha+1)}\circ (P-z^m)^{-1} f\|_{L^{\frac{2n}{n-m}}(M)}\le C \alpha^{m-1}(\sup_{\tau\in  [\alpha-1, \alpha+1)} |(\tau^m-z^m)^{-1}|)\|f\|_{L^{\frac{2n}{n+m}}(M)},
\end{equation}
Writing
\[
z^m=(\alpha+i\beta)^m=\alpha^m(1+mi\beta/\alpha+\mathcal{O}(\beta^2/\alpha^2)),
\]
we have
\begin{equation}
\label{eq_4_1_3}
\textrm{Im } z^m=m\beta\alpha^{m-1}+\mathcal{O}(\beta^2\alpha^{m-2})\ge \frac{m}{2}\beta\alpha^{m-1}\ge \frac{m}{2}\delta\alpha^{m-1},
\end{equation}
for $\alpha$ sufficiently large. Therefore, it follows from \eqref{eq_4_1_2}, \eqref{eq_4_1_3} and \eqref{eq_resolvent_est} that
\begin{equation}
\label{eq_4_1_4}
\|(I-\chi_{[\alpha-1, \alpha+1)})\circ (P-z^m)^{-1} f\|_{L^{\frac{2n}{n-m}}(M)}\le C\|f\|_{L^{\frac{2n}{n+m}}(M)},
\end{equation}
for all $z\in \Xi$, $\textrm{Re }z=\alpha\gg 1$, and $\delta\le \textrm{Im }z\le 1$, uniformly in $z$.

Let $z\in \Xi$, $\textrm{Re }z=\alpha\gg 1$, and $0<\textrm{Im }z=\beta\le 1/2$. Then using the fact that $\alpha+i\in \Xi$ for $\alpha$ sufficiently large and  \eqref{eq_4_1_4}, we see that
 \eqref{eq_4_1_1} follows once we establish  that
\begin{equation}
\label{eq_4_1_4_0}
\|(I-\chi_{[\alpha-1, \alpha+1)})\circ ((P-z^m)^{-1}-(P-(\alpha+i)^m)^{-1}) f\|_{L^{\frac{2n}{n-m}}(M)}\le C\|f\|_{L^{\frac{2n}{n+m}}(M)},
\end{equation}
uniformly in $z$.
We have
\begin{equation}
\label{eq_4_1_4_3}
\begin{aligned}
(I&-\chi_{[\alpha-1, \alpha+1)})\circ ((P-z^m)^{-1}-(P-(\alpha+i)^m)^{-1}) f\\
&=\bigg(\sum_{\mu_j\in [0,\alpha-1)} +\sum_{\mu_j\in [\alpha+1,+\infty)}\bigg) \bigg(\frac{1}{\mu_j^m-z^m}-\frac{1}{\mu_j^m-(\alpha+i)^m}\bigg) E_jf\\
&= \bigg(\sum_{\mu_j\in [0,\alpha-1)} +\sum_{k=2}^\infty \sum_{\mu_j\in [\alpha+k-1,\alpha+k)}\bigg) \bigg(\frac{1}{\mu_j^m-z^m}-\frac{1}{\mu_j^m-(\alpha+i)^m}\bigg) E_jf.
\end{aligned}
\end{equation}
By \eqref{eq_truncated}, for $k=2,3\dots$, we get
\begin{equation}
\label{eq_4_1_5}
\begin{aligned}
\|\sum_{\mu_j\in [\alpha+k-1,\alpha+k)}\bigg(\frac{1}{\mu_j^m-z^m}-\frac{1}{\mu_j^m-(\alpha+i)^m}\bigg) E_jf\|_{L^{\frac{2n}{n-m}}(M)}\le C(\alpha+k)^{m-1}\\
\sup_{\tau\in [\alpha+k-1,\alpha+k)} \bigg| \frac{z^m-(\alpha+i)^m}{(\tau^m-z^m)(\tau^m-(\alpha+i)^m)}\bigg|\|f\|_{L^{\frac{2n}{n+m}}(M)}.
\end{aligned}
\end{equation}
We have, for $\alpha$ sufficiently large, that
\[
z^m-(\alpha+i)^m=\alpha^{m-1}m i(\beta-1)+\mathcal{O}(\alpha^{m-2}),
\]
and therefore,
\begin{equation}
\label{eq_4_1_6}
|z^m-(\alpha+i)^m|\le C\alpha^{m-1}.
\end{equation}
As $\textrm{Re }z^m=\alpha^m+\mathcal{O}(\alpha^{m-2})$, we  obtain that
\begin{equation}
\label{eq_4_1_7}
\begin{aligned}
&|\tau^m-z^m|\ge |\tau^m-\alpha^m-\mathcal{O}(\alpha^{m-2})|\\
&= |(\tau-\alpha)(\tau^{m-1}+\tau^{m-2}\alpha+\dots+\tau\alpha^{m-2}+\alpha^{m-1})-\mathcal{O}(\alpha^{m-2})|\\
&\ge  (k-1)(\tau^{m-1}+\alpha^{m-1})-|\mathcal{O}(\alpha^{m-2})|\ge (k-1)\tau^{m-1}\ge
 (k-1)(\alpha+k)^{m-1}/C,
\end{aligned}
\end{equation}
for $\tau\in  [\alpha+k-1,\alpha+k)$, $k=2,3,\dots$, and  $\alpha$ sufficiently large. Thus, it follows from \eqref{eq_4_1_5}, \eqref{eq_4_1_6}, and \eqref{eq_4_1_7} that
\begin{equation}
\label{eq_4_1_4_1}
\begin{aligned}
\|\sum_{\mu_j\in [\alpha+k-1,\alpha+k)}\bigg(\frac{1}{\mu_j^m-z^m}-\frac{1}{\mu_j^m-(\alpha+i)^m}\bigg) E_jf\|_{L^{\frac{2n}{n-m}}(M)}\\
\le \frac{C}{(k-1)^2} \|f\|_{L^{\frac{2n}{n+m}}(M)},
\end{aligned}
\end{equation}
for $k=2,3,\dots$.
Using  \eqref{eq_truncated} and rescaling, we get
\begin{equation}
\label{eq_4_1_4_2}
\|\sum_{\mu_j\in [0,\alpha-1)}\bigg(\frac{1}{\mu_j^m-z^m}-\frac{1}{\mu_j^m-(\alpha+i)^m}\bigg) E_jf\|_{L^{\frac{2n}{n-m}}(M)}
\le C \|f\|_{L^{\frac{2n}{n+m}}(M)}.
\end{equation}
Hence, \eqref{eq_4_1_4_0} follows from  \eqref {eq_4_1_4_3},  \eqref {eq_4_1_4_1}, and  \eqref {eq_4_1_4_2}. The proof of \eqref{eq_4_1_1} is complete.

Let us now show  \eqref{eq_4_1_1_another}. To that end,  letting $w=ze^{-2\pi i/m}$, we have $w^m=z^m$, and therefore, \eqref{eq_4_1_1_another} is a consequence of  the uniform estimate,
\[
\|(I-\chi_{[\alpha-1, \alpha+1)})\circ ((P-w^m)^{-1}-(P-(\alpha+i)^m)^{-1}) f\|_{L^{\frac{2n}{n-m}}(M)}\le C\|f\|_{L^{\frac{2n}{n+m}}(M)},
\]
with $z\in \Xi$, $w=ze^{-2\pi i/m}$, $\textrm{Re }w=\alpha\gg 1$, and $0<-\textrm{Im } w\le 1$. This is obtained  similarly to the derivation of \eqref{eq_4_1_4_0}. The proof of  Lemma \ref{lem_satur_1} is complete.
\end{proof}

Let
\[
N(\alpha)=\#\{j:\mu_j<\alpha\}
\]
be the counting function for the eigenvalues of the operator $Q$.  We have
\begin{equation}
\label{eq_exp_weyl_fun}
N(\alpha)=\int_M S_\alpha(x,x) d\mu(x),
\end{equation}
where
\[
S_\alpha(x,y)= \sum_{\mu_j< \alpha} e_j(x)\overline{e_j(y)}
\]
is the spectral function.

Similarly to \cite[Theorem 1.2]{Bourgain_Shao_Sogge_Yao} we obtain the following result which gives a sufficient condition for the optimality of the region $\Xi_\delta$ in the uniform resolvent estimate \eqref{eq_resolvent_est} for operators of order $m$,  in terms of the density of eigenvalues in shrinking  intervals of the form $[\alpha_k-\beta_k,\alpha_k+\beta_k)$, $\alpha_k\to \infty$, $0<\beta_k\to 0$ as $k\to \infty$.

\begin{lem}
\label{lem_satur_2}
 Assume that there exist sequences $\alpha_k\to \infty$ and $0<\beta_k\to 0$ as $k\to \infty$ such that
\begin{equation}
\label{eq_4_1_8}
(\beta_k\alpha_k^{n-1})^{-1}[N(\alpha_k+\beta_k)-N(\alpha_k-\beta_k)]\to \infty,\quad k\to \infty.
\end{equation}
 Let $z_k^{(1)}=\alpha_k+i\beta_k$ and $z_k^{(2)}=e^{2\pi i/m}(\alpha_k-i\beta_k)$. Then we have
\begin{equation}
\label{eq_4_1_9}
\|(P-(z_k^{(j)})^m)^{-1}\|_{L^{\frac{2n}{n+m}}(M)\to L^{\frac{2n}{n-m}}(M)}\to \infty,\quad k\to \infty, \quad j=1,2.
\end{equation}
\end{lem}

\begin{proof}
In what follows we shall only establish \eqref{eq_4_1_9} for $j=1$, the proof in the other case being similar. We shall then write $z_k=z_k^{(1)}$.  Let us notice  that $z_k\in \Xi$ for $k$ large enough.

By \eqref{eq_4_1_1}, we know that for large $k$,
\[
\|(I-\chi_{[\alpha_k-1,\alpha_k+1)})\circ (P-z_k^m)^{-1}\|_{L^{\frac{2n}{n+m}}(M)\to L^{\frac{2n}{n-m}}(M)}\le C,
\]
uniformly in $k$.  Thus, we only need to show that
\begin{equation}
\label{eq_4_1_10}
\|\chi_{[\alpha_k-1,\alpha_k+1)}\circ (P-z_k^m)^{-1}\|_{L^{\frac{2n}{n+m}}(M)\to L^{\frac{2n}{n-m}}(M)}\to \infty,\quad k\to \infty.
\end{equation}

Let $g\in C^\infty_0(\R)$ be such that $0\notin\supp(g)$ and $g(\tau)=1$ for $\tau\in [1/2,2]$. Then for large $k$, we have
\begin{equation}
\label{eq_4_1_11}
\chi_{[\alpha_k-1,\alpha_k+1)}=g(Q/\alpha_k)\circ \chi_{[\alpha_k-1,\alpha_k+1)} \circ g(Q/\alpha_k).
\end{equation}
Using \eqref{eq_4_1_11}  and Lemma \ref{lem_Bernstein_type}, we obtain
\begin{align*}
\|&\chi_{[\alpha_k-1,\alpha_k+1)}\circ (P-z_k^m)^{-1}f\|_{L^\infty(M)}\\
&=\|g(Q/\alpha_k)\circ  \chi_{[\alpha_k-1,\alpha_k+1)}\circ (P-z_k^m)^{-1} \circ g(Q/\alpha_k)f\|_{L^\infty(M)}\\
&\le
C\alpha_k^{\frac{n-m}{2}} \|\chi_{[\alpha_k-1,\alpha_k+1)}\circ (P-z_k^m)^{-1}\|_{L^{\frac{2n}{n+m}}(M)\to L^{\frac{2n}{n-m}}(M)}\| g(Q/\alpha_k)f\|_{L^{\frac{2n}{n+m}}(M)}\\
&\le C\alpha_k^{n-m} \|\chi_{[\alpha_k-1,\alpha_k+1)}\circ (P-z_k^m)^{-1}\|_{L^{\frac{2n}{n+m}}(M)\to L^{\frac{2n}{n-m}}(M)}\|f\|_{L^1(M)}.
\end{align*}
Thus, in order to show \eqref{eq_4_1_10}  it suffices to check that
\begin{equation}
\label{eq_4_1_11_0}
\alpha_k^{-(n-m)} \|\chi_{[\alpha_k-1,\alpha_k+1)}\circ (P-z_k^m)^{-1}\|_{L^1(M)\to L^\infty(M)}\to \infty,\quad k\to \infty.
\end{equation}
The kernel of the operator $\chi_{[\alpha_k-1,\alpha_k+1)}\circ (P-z_k^m)^{-1}$ is given by
\[
K(x,y)=\sum_{\mu_j\in [\alpha_k-1,\alpha_k+1)}\frac{1}{\mu_j^m-z_k^m} e_j(x)\overline{e_j(y)}.
\]
We have
\begin{align*}
\alpha_k^{-(n-m)}& \|\chi_{[\alpha_k-1,\alpha_k+1)}\circ (P-z_k^m)^{-1}\|_{L^1(M)\to L^\infty(M)}=\alpha_k^{-(n-m)} \textrm{sup}_{x,y\in M}|K(x,y)|\\
&\ge \alpha_k^{-(n-m)}  \sup_{x\in M}\bigg| \sum_{\mu_j\in [\alpha_k-1,\alpha_k+1)}\frac{1}{\mu_j^m-z_k^m} |e_j(x)|^2 \bigg|\\
&\ge \alpha_k^{-(n-m)} \sup_{x\in M}\bigg|  \textrm{Im } \sum_{\mu_j\in [\alpha_k-1,\alpha_k+1)}\frac{\mu_j^m -\overline{z_k}^m}{|\mu_j^m-z_k^m|^2} |e_j(x)|^2 \bigg|\\
&\ge  \alpha_k^{-(n-m)}|\textrm{Im } (-\overline{z_k}^m)| \sup_{x\in M}\sum_{\mu_j \in [\alpha_k-\beta_k,\alpha_k+\beta_k)}\frac{1}{|\mu_j^m-z_k^m|^2} |e_j(x)|^2:=L_k,
\end{align*}
for $k$ sufficiently large.
Writing $\overline{z_k}^m=(\alpha_k-i\beta_k)^m$, we get
\begin{equation}
\label{eq_4_1_12}
\textrm{Im } (-\overline{z_k}^m)=m\beta_k\alpha_k^{m-1}+\mathcal{O}(\beta_k^2\alpha_k^{m-2})\ge m\beta_k\alpha_k^{m-1}/2,
\end{equation}
for $k$ sufficiently large. Using the fact that $\mu_j \in [\alpha_k-\beta_k,\alpha_k+\beta_k)$ in the last sum, we obtain that
\begin{equation}
\label{eq_4_1_13}
|\mu_j^m-z_k^m|=|\mu_j-z_k||\mu_j^{m-1}+\mu_j^{m-2}z_k+\dots+ \mu_jz_k^{m-2}+z_k^{m-1}|\le C\beta_k\alpha_k^{m-1},
\end{equation}
for $k$ sufficiently large.  It follows from \eqref{eq_exp_weyl_fun}, \eqref{eq_4_1_12}, \eqref{eq_4_1_13} and \eqref{eq_4_1_8} that
\begin{align*}
L_k&\ge \frac{1}{C}(\beta_k\alpha_k^{n-1})^{-1} \sup_{x\in M}\sum_{\mu_j \in [\alpha_k-\beta_k,\alpha_k+\beta_k)} |e_j(x)|^2\\
&\ge  \frac{1}{C}(\beta_k\alpha_k^{n-1})^{-1} \frac{1}{\textrm{Vol}(M)} \int_M \sum_{\mu_j \in [\alpha_k-\beta_k,\alpha_k+\beta_k)} |e_j(x)|^2 d\mu(x)\\
&= \frac{1}{C}(\beta_k\alpha_k^{n-1})^{-1} \frac{1}{\textrm{Vol}(M)} [N(\alpha_k+\beta_k)-N(\alpha_k-\beta_k)]\to \infty,
\end{align*}
as $k\to \infty$.  Hence, we get \eqref{eq_4_1_11_0}, which completes the proof of \eqref{eq_4_1_9}.
 The proof of Lemma \ref{lem_satur_2} is complete.
\end{proof}

Notice that the Weyl law, see \cite{Hormander_1968},
\[
N(\alpha)=C\alpha^n+\mathcal{O}(\alpha^{n-1}),\quad C=(2\pi)^{-n}\int\!\!\!\int_{\{(x,\xi)\in T^*M:q(x,\xi)\le 1\}} dxd\xi,
\]
implies that
\[
N(\alpha_k+1)-N(\alpha_k-1)=\mathcal{O}(\alpha_k^{n-1}).
\]
Consequently, to find sequences $\alpha_k\to \infty$ and $0<\beta_k\to 0$ as $k\to \infty$ satisfying  \eqref{eq_4_1_8}, we would like to exhibit  a situation when the spectrum of  the operator $Q$ is distributed in a non-uniform fashion, clustering around the sequence $\alpha_k$.

To verify the  assumption \eqref{eq_4_1_8} in Lemma \ref{lem_satur_2},  we shall need the following result concerning the spectrum of $Q$, when the Hamilton flow of $q$ is periodic,  due to \cite{Weinstein_1977} and \cite{Colin_de_Verdiere_1979}, see also \cite[Theorem 29.2.2]{Hormander_book_4}.

\begin{thm}
\label{thm_weinstein}
Let $Q\in \Psi_{\emph{\textrm{cl}}}^1(M)$ be positive elliptic self-adjoint operator with principal symbol $q$ and zero subprincipal symbol. Assume that the Hamilton flow $\exp(t H_q)$, generated by the principal symbol $q$, is periodic with a common minimal period $T$ on $q^{-1}(1)$. Then there is a constant $C>0$ such that all eigenvalues of $Q$, except finitely many, belong to the intervals $I_k:=[\frac{2\pi}{T}(k+\frac{\alpha}{4})-\frac{C}{k}, \frac{2\pi}{T}(k+\frac{\alpha}{4})+\frac{C}{k}]$, $k=1,2\dots$, where $\alpha>0$ is a constant. Furthermore, the number of eigenvalues of $Q$ in $I_k$, denoted by $d_k$,   is a polynomial in $k$ of degree $n-1$ of the form
\[
d_k=nk^{n-1} T^{-n}\int\!\!\!\int_{q<1}dxd\xi+\mathcal{O}(k^{n-2}).
\]
\end{thm}

To prove Theorem \ref{thm_main_2},  let $Q=P^{1/m}$ and observe that the subprincipal symbol of $Q$ vanishes, see \cite[Section 1]{Duistermaat_Guillemin}.  It follows from  Theorem \ref{thm_weinstein} that the assumptions of Lemma \ref{lem_satur_2} are satisfied with $\alpha_k=\frac{2\pi}{T}(k+\frac{\alpha}{4})$ and  $C/k<\beta_k\to 0$ as $k\to \infty$.   The proof of Theorem \ref{thm_main_2}  is complete.

\section*{Acknowledgements}

The research of K.K. is partially supported by the
Academy of Finland (project 255580). K.K. would like to thank Kari Vilonen for some helpful advice.  The research of
G.U. is partially supported by the National Science Foundation and the Fondation de Sciences Math\'ematiques de Paris.

\end{document}